\documentclass[reqno, english]{amsart}
\usepackage{etex}
\usepackage{amsmath,amssymb,amsthm,bbm,mathtools,comment}
\usepackage[shortlabels]{enumitem}
\usepackage[pdftex,colorlinks,backref=page,citecolor=blue]{hyperref}
\hypersetup{pdfpagemode=UseNone,pdfstartview={XYZ null null 1.00}}
\usepackage[mathscr]{euscript}
\usepackage[usenames,dvipsnames]{color}
\usepackage{adjustbox,tikz,calc,graphics,babel,standalone}
\usetikzlibrary{shapes.misc,calc,intersections,patterns,decorations.pathreplacing}
\usetikzlibrary{arrows,shapes,positioning}
\usetikzlibrary{decorations.markings}
\usepackage[final]{microtype}
\usepackage[numbers]{natbib}
\usepackage{amsfonts}
\usepackage{caption}
\usepackage{subcaption}
\usepackage{verbatim}
\usepackage{array}
\usepackage[frame,cmtip,arrow,matrix,line,graph,curve]{xy}
\usepackage{pifont}
\usepackage[final]{microtype}
\usepackage{todonotes}
\usepackage{thm-restate}
\usepackage{polski}
\allowdisplaybreaks

\theoremstyle{plain}
\newtheorem{theorem}{Theorem}[section]		
\newtheorem{lemma}[theorem]{Lemma}
\newtheorem{claim}[theorem]{Claim}
\newtheorem{proposition}[theorem]{Proposition}
\newtheorem{corollary}[theorem]{Corollary}
\newtheorem{observation}[theorem]{Observation}
\newtheorem{conjecture}[theorem]{Conjecture}
\newtheorem{problem}[theorem]{Problem}
\newtheorem{definition}[theorem]{Definition}
\theoremstyle{remark}
\newtheorem*{remark}{Remark}

\def\CC{\mathcal{C}}
\def\EE{\mathcal{E}}

\def\FF{\mathcal{F}}
\def\HH{\mathcal{H}}
\def\II{\mathcal{I}}
\def\GG{\mathcal{G}}
\def\PP{\mathcal{P}}
\def\E{\mathbb{E}}
\renewcommand{\P}{\mathbb{P}}

\DeclareMathOperator\ex{ex}
\newcommand{\eps}{\ensuremath{\varepsilon}}
\let\emptyset\varnothing

\newcommand{\hide}[1]{}
\newcommand{\zh}[1]{\textcolor{blue}{ZH: #1}}

\let\originalleft\left
\let\originalright\right
\renewcommand{\left}{\mathopen{}\mathclose\bgroup\originalleft}
\renewcommand{\right}{\aftergroup\egroup\originalright}

\makeatletter
\def\imod#1{\allowbreak\mkern10mu({\operator@font mod}\,\,#1)}
\makeatother

\author{Ant\'onio Gir\~ao}
\address[Gir\~ao]{Mathematical Institute, University of Oxford, Andrew Wiles Building, Radcliffe Observatory Quarter, Woodstock Road, Oxford, UK.}
\email{girao@maths.ox.ac.uk}

\author{Zach Hunter}
\address[Hunter]{Department of Mathematics, ETH, Z\"urich, Switzerland.}
\email{zach.hunter@math.ethz.ch}

\date{\today} 

\begin{document}

\thanks{AG was supported by EPSRC grant EP/V007327/1.}

\title{Induced subdivisions in $K_{s,s}$-free graphs with polynomial average degree}
\maketitle
\begin{abstract}
    In this paper we prove that for every $s\geq 2$ and every graph $H$ the following holds. Let $G$ be a graph with average degree $\Omega_H(s^{C|H|^2})$, for some absolute constant $C>0$, then $G$ either contains a $K_{s,s}$ or an induced subdivision of $H$. This is essentially tight and confirms a conjecture of Bonamy, Bousquet, Pilipczuk, Rz\k{a}\.zewski, Thomass\'e, and Walczak in \cite{bonamy}. A slightly weaker form of this has been independently proved by Bourneuf, Buci\'c, Cook, and Davies~\cite{Bucicetal}.
    
    We actually prove a much more general result which implies the above (with worse dependence on $|H|$). We show that for every $ k\geq 2$ there is $C_k>0$ such that any graph $G$ with average degree $s^{C_k}$ either contains a $K_{s,s}$ or an induced subgraph $G'\subseteq G$ without $C_4$'s and with average degree at least $k$. 
    
    Finally, using similar methods we can prove the following. For every $k,t\geq 2$ every graph $G$ with average degree at least $C_tk^{\Omega(t)}$ must contain either a $K_k$, an induced $K_{t,t}$ or an induced subdivision of $K_k$. This is again \textit{essentially} tight up to the implied constants and answers in a strong form a question of Davies. 
\end{abstract}
\section{Introduction}

The study of the unavoidable substructures in graphs of high chromatic number is a well known area of graph theory. 
A graph can have high chromatic number due to the presence of a large clique so it is natural to ask for which families $\mathcal{H}$ of graphs does there exist a function $f:\mathbb{N} \rightarrow \mathbb{N}$ with the property that any graph with clique number at most $k$ and chromatic number at least $f(k)$ must contain an induced subgraph isomorphic to a graph $H\in \mathcal{H}$. 

More generally, let $\mathcal{H}$ be a hereditary class of graphs then we say $\mathcal{H}$ is $\chi$-bounded if there is a function $f: \mathbb{N} \rightarrow \mathbb{N}$ such that for every $G\in \mathcal{H}$ satisfies $\chi(G)\leq f(\omega(G))$, where $\omega(G)$ is the order of the largest clique in $G$. Finally we say $\mathcal{H}$ is polynomially $\chi$-bounded if $f$ can be taken to be a polynomial. 

As usual, given a family of graphs $\mathcal{F}$, we say $G$ is $\mathcal{F}$-free if there is no induced subgraph of $G$ isomorphic to a graph in $\mathcal{F}$. 
A classical result of Erd\H{o}s states that all finite families $\mathcal{F}$ for which the class of $\mathcal{F}$-free graphs is $\chi$-bounded must contain a tree and it is a long standing open conjecture due Gy\'arf\'as~\cite{gyarfasconj} and independently by Sumner~\cite{Sumner} that the converse holds. Namely, that for every $T$, the family of $T$-free graphs is $\chi$-bounded. 

When the family of excluded graphs is not finite the situation seems to be more delicate. A remarkable result of Scott and Seymour \cite{HolesResidues} (building on previous work with Chudnovsky and Spirkl~\cite{longholes,longoddholes}) confirms in a very strong form several conjectures of Gy\'arf\'as~\cite{G87}. It says that for every $m,n \in \mathbb{N}$ the family of graphs without an induced cycle of length congruent with $m$ (modulo $n$) is $\chi$-bounded. One might then be tempted to conjecture that the class of graphs avoiding all induced subdivisions of a graph $H$ is $\chi$-bounded but this is false as shown by Pawlik, Kozik, Krawczyk, Laso\'n, Micek, Trotter, and Walczak \cite{counterex} who constructed for every $k\geq 2$ a family of line segments in the plane for which its \textit{intersection graph} is triangle-free and has chromatic number greater than $k$. It is not hard to see that any \textit{proper} (i.e. all subdivided paths have length at least $2$) induced subdivision of a non-planar graph cannot be represented as an intersection graph of line segments in the plane.

A lovely result of  Bria\'nski, Davies, and Walczak~\cite{separatingpoly} (extending ideas of Carbonero, Hompe, Moore, and Spirkl \cite{CHMS23}) shows there are $\chi$-bounded families of graphs for which the growth rate of $f$ is arbitrarily high which sparks the broader question of when a $\chi$-bounded family is polynomially $\chi$-bounded. This question has attracted a lot of attention (see e.g. \cite{survey}) but a full classification seems to be out of reach. 
Somewhat surprisingly, we do not even know whether the family of graphs avoiding an induced fixed path of length $t\geq 5$ is polynomially $\chi$-bounded (as noted by Trotignon and Pham \cite{TP}).

A similar notion to $\chi$-boundedness, but with average degree instead of chromatic number was recently introduced. Intuitively, while a class is $\chi$-bounded if cliques are the only thing that can force large chromatic number, it is ``degree-bounded'' if, instead, balanced bicliques are the only thing that can force large average degree. 
Formally, we say that a hereditary family $\mathcal{F}$ is \textit{degree-bounded} if there exists a function $g: \mathbb{N} \rightarrow \mathbb{N}$ such that for every $G\in \mathcal{F}$, we have $d(G) \leq g(\tau(G))$, where we write $d(G)$ for the average degree of $G$ and $\tau(G)$ for the \textit{biclique number} of $G$, which is the largest integer $s$ so that $G$ has a (not necessarily induced) copy of $K_{s,s}$. Any such function $g$ is called a \textit{degree-bounding function} for the class.

An important result in this area due to K\"uhn and Osthus~\cite{KO1} states that for every graph $H$, the class of graphs which do not contain an induced subdivision of $H$ is degree-bounded. 
\begin{theorem}[K\"uhn and Osthus]
    For every graph $H$ and integer $s$, there is an integer $p(s,H)$ such that every graph $G$ without a $K_{s,s}$ and with average degree at least $p(s,H)$ contains an induced subdivision of $H$. 
\end{theorem}
Their bounds for $p(s,H)$ are roughly triply exponential in $s$, for fixed $H$. A natural conjecture raised by Bonamy et al.~\cite[Conjecture~33]{bonamy} asserts that actually $p(s,H)$ could be taken to be a polynomial in $s$. Some partial results towards this conjecture are known. Indeed, Scott, Seymour, and Spirkl~\cite{SSS} established a quantitative strengthening of a result of Kierstead and Penrice \cite{KP}, by showing that the class of $T$-induced-free graphs is polynomially degree bounded (which of course implies a polynomial upper bound for $p(s,T)$). This in turn generalized another theorem of Bonamy, Bousquet, Pilipczuk, Rz\k{a}\.zewski,
Thomass\'e and Walczak \cite{bonamy}, who proved the same result when $T$ is a path.

Our first theorem confirms this conjecture, in a very strong form.

\begin{theorem}\label{thm: main1}
    For each integer $h$, and all large $s$ the following holds. 
    
    Let $G$ be a graph with $d(G)\geq s^{500h^2}$. Then $G$ either contains a (not necessarily induced) $K_{s,s}$ or an induced proper subdivision of $K_h$. In fact, we ensure this subdivision is \textit{balanced} (meaning each edge of $K_h$ is replaced by a path of some common length, $\ell$). 
\end{theorem}

By taking a random graph $G \sim G(N,p)$ on $ N:=s^{h^2/100}$ vertices (for\footnote{Occasionally we will write `$x\ggg y$' to informally say that $x$ is sufficiently big compared to $y$.} $s\ggg h$) with $p=1-\frac{h^2\log s}{s}\geq 1/2$, we have with positive probability that $d(G)\geq s^{h^2/100}/4$ and $G$ does not contain a $K_{s, s}$ nor an independent set of size $h^2/4$ (which implies there is no proper induced subdivision of a $K_h$). This shows the theorem is tight as a function of $s$. It is also clear that any proper induced subdivision of a $K_h$ contains an induced subdivision of every graph $H$ with $|H|\leq h$.   

In other words, Theorem~\ref{thm: main1} says that the class of graphs without an induced subdivision of a fixed graph $H$ is degree-bounded with a polynomial degree-bounding function. It is therefore natural to ask whether the same phenomenon holds for every such class i.e. whether every degree-bounded class of graphs has a polynomial degree-bounding function. In a very recent paper \cite{c4free} the authors together with Du, McCarty and Scott proved that every degree-bounded family of graphs is essentially exponentially degree-bounded very much in contrast to the $\chi$-boundedness case. 
More precisely, it was shown that for every hereditary degree-bounded class of graphs $\mathcal{F}$, there exists a constant $C_{\mathcal{F}}$ so that $(C_{\mathcal{F}})^{s^3}$ is a degree-bounding function for $\mathcal{F}$. 

Our strongest result establishes a polynomial bound for the degree-bounding function of every degree-bounded hereditary class of graphs which follows straightforwardly from the following stronger statement. 
\begin{theorem}\label{thm:main2}
Fix any $k\geq 2$ and suppose $s$ is sufficiently large. Any graph $G$ with average degree $\ge s^{5000k^4}$ either contains a $K_{s,s}$ (not necessarily induced) or an induced subgraph $G'\subseteq G$ with no $C_4$ and $d(G)\geq k$. 
\end{theorem}
\begin{remark}
    The exponent `$5000k^4$' is of roughly the right shape (being polynomial in $k$). Indeed, by considering $G \sim G(s^{k^2/100},1-1/s^{9/10})$ for large enough $s$, one gets that the exponent must be at least $k^2/100$.
\end{remark}

We turn now to a recent question asked by James Davies \cite{Davies1} on $\chi$-boundedness. 
\begin{problem}\label{Prob: Davies}
    For $t \ge 1$, let $\CC$ be the family of graphs of without an induced subdivision of $K_{2,t}$ (not necessarily proper). Is this family polynomially $\chi$-bounded? 
\end{problem}
An affirmative answer follows as a simple corollary of Theorem~\ref{thm: main1}. Indeed, let $\omega \coloneqq \omega(G)$. We shall show that for every $G\in \CC$, $d(G)={(t\omega)}^{O(t^2)}$. Observe that if $G$ contains a $K_{\omega+1, {(t\omega(G))}^{t}}$ then by Ramsey' Theorem, we must have an independent edge joined to an independent set of size $t$ which forms an induced $K_{2,t}$. If not, applying Theorem~\ref{thm: main1} with $H=K_{2,t}$ and $s=(t\omega)^{t}$, we must have an induced subdivision of $K_{2,t}$, as required. 

We can actually show a much stronger result which answers Problem~\ref{Prob: Davies} with essentially tight bounds. 

As before, we say a subdivision is proper if every edge is subdivided at least once and further we say a subdivision is balanced if all paths have the same length. 
 
\begin{theorem}\label{thm: main}
For every $s,t \in \mathbb{N}$ with $s\le t$, the following holds for all large $k$. Let $G$ be a graph with $ d(G)\ge k^{4000t}$, then either:
\begin{itemize}
    \item $G$ contains a $k$-clique;
    \item $G$ contains an induced proper balanced subdivision of $K_h$, with $h\ge d(G)^{1/(10^{9}s)}$;
    \item $G$ contains an induced $K_{s,t}$.
\end{itemize}
\end{theorem}
\begin{remark}
    We did not try to optimize the constants `$4000$' and `$1/10^{9}$', here (nor the constants `$500$' and `$5000$' from our last two theorems).  
\end{remark}

We note that this result is \textit{essentially} sharp (up to the constant in the exponent) when $k\ggg t$. Indeed, if $G$ is a graph without $k$-cliques or independent sets of size $t$, then it will also lack an induced proper subdivision of $K_t$; and a simple probabilistic construction shows that there exist such $G$ with average degree $\ge k^{\Omega(t)}$ (cf. \cite[Proof of Corollary~6.9]{c4free}). Meanwhile, if $G$ does not contain $K_{s,s}$ as a subgraph, then it is $K_{2s}$-free and $K_{s,s}$-free, and if it lacks independent sets of size $h$ it will lack a proper subdivision of $K_h$; another probabilistic construction shows that there are such graphs with average degree $\ge h^{\Omega(s)}$ (cf. \cite{c4free}).

Also, the above result becomes false if we remove either of the first two bullets (by either taking $G$ to be an arbitrarily large clique, or a graph with no $C_4$ of arbitrarily large degree). Moreover, if we removed the last bullet, then we know that $G$ could be triangle-free with arbitrarily high average degree and without an induced proper subdivision of $K_5$ (due to \cite{counterex}). 

\begin{remark}
    \hide{A weaker version of Theorem~\ref{thm: main1} was proved independently by Bourneuf, Buci\'c, Cook and Davies \cite{Bucicetal}. The same authors also show a version of Theorem~\ref{thm:main2} which is quantitatively far from optimal. }
   
    In independent contemporaneous work \cite{Bucicetal}, Bourneouf, Buci\'c, Cook and Davies have obtained some results that overlap with those proved here. Firstly, they proved a version of Theorem~\ref{thm: main1} (cf. \cite[Theorem~1.2]{Bucicetal}), but the dependence on $H$ is not established. Secondly, they showed a reduction (cf. \cite[Theorem~1.3]{Bucicetal}) which made some partial progress towards our main theorem (Theorem~\ref{thm:main2}).    
    Lastly, Lemma~\ref{one sided erdos hajnal}, a key lemma of ours, is a quantitatively explicit version of \cite[Theorem~1.5]{Bucicetal}.   

    While both this paper and \cite{Bucicetal} build upon previous techniques of \cite{c4free}, the intermediate steps are quite different.
\end{remark}

\section{Preliminaries}
\subsection{Notation}
Throughout this paper, for positive integer $n$ we write $[n]:=\{1,2,\dots,n\}$. 

Given a multigraph $G$, we denote $d(G)$ to be the average degree of $G$ (i.e. $d(G)\coloneqq \frac{2|E(G)|}{|G|}$). Similarly for $s$-uniform multihypergraphs $\HH$ we write $d(\HH):= \frac{s|E(G)|}{|V(\HH)|}$.

For a family of graphs $\mathcal{F}$, we say $G$ is $\mathcal{F}$-free if for all $F\in \mathcal{F}$, $G$ does not contain any copy of $F$ as an induced subgraph. 
As usual, we denote $\omega(G)$ to be the size of the largest clique in $G$ and $e(G):= |E(G)|$. Moreover, for a subset $B\subset V(G)$, we let
 $d_B(x)=|N_B(x)|$ where $N_B(x)=\{ y\in B: y\in N(x)\}$. Finally, we denote $G[A,B]$ to be the bipartite graph induced between $A$ and $B$. 

\subsection{Tools}

We shall need the classical Ramsey's Theorem. 
\begin{theorem}
    $R(s,k)\leq \binom{s+k}{s}=O_s(k^s)$. 
\end{theorem}
\noindent Next we recall a standard consequence of dependent random choice.
\begin{lemma}\label{DRC}
    Let $G$ be graph with $n$ vertices, and consider some integer $s\ge 1$. 
    
    If $\#(x\in V(G):d(x)\ge n^{1-1/100s}) \ge \sqrt{n}$, then $G$ has a set $S$ of $n^{1/3}-1$ vertices so that any $s$-subset $S'\subset S$ has a common neighborhood of size at least $n^{0.9}$.
\end{lemma}
\begin{proof}
    Let $S_* := \{x\in V(G): d(x)\ge n^{1-1/100s}\}$. Now pick vertices $v_1,\dots,v_{10s}\sim V(G)$ uniformly at random (independently).

    Take $S_0 := \bigcap_{i=1}^{10s} N(v_i)$. For $x\in S_*$, we have that $\P(x\in S_0)= (d(x)/n)^{10s}\ge n^{-1/10}$, thus $\E[|S_0|] \ge |S_*|n^{-1/10}\ge n^{1/3}$. Meanwhile, we have that \[\E[\#(x_1,\dots,x_s\in S_0^s: |N(x_1)\cap \dots \cap N(x_s)|\le n^{0.9})] \le n^s(n^{0.9}/n)^{10s} = 1.\]Deleting a vertex from such bad $s$-tuple gives the result.
\end{proof}
\noindent Putting these together yields the following supersaturation result.
\begin{corollary}\label{supersat}
    Let $t\ge s \ge 2$. For all sufficiently large $k$ with respect to $s,t$ the following holds.
    Let $G$ be a graph on $n$ vertices which is $K_{s,t}$-free and contains no $K_k$. Then, $G$ has at least $\Omega_s(n^s)$ independent sets of size $s$, provided $n\ge k^{100t}$.

\end{corollary}
    \begin{proof}
        If $\frac{e(G)}{\binom{n}{2}}\le 1/s^2$, then sampling a $s$-subset $S\in \binom{V(G)}{s}$ uniformly at random, we will have that \[\E[e(G[S])] \le \frac{\binom{s}{2}}{s^2}<1/2.\]Thus, with probability $\ge 1/2$, $S$ is an independent set. This implies that there are at least $\frac{1}{2}\binom{n}{s}\ge \frac{n^s}{3(s!)}$ independent sets (assuming $n$ is sufficiently large).

        Now suppose this was not the case. Then, for sufficiently large $n$, we must have that $e(G)\ge 2n^{2-1/100s}$. This easily implies that $\#(x\in V(G): d(x)\ge n^{1-1/100s})\ge \sqrt{n}$. Applying Lemma~\ref{DRC}, we can find $S\subset V(G)$ with $|S|\ge n^{1/4}$, so that every $s$-subset $S'\subset S$ has a common neighborhood of size at least $ \sqrt{n}$. 

        Meanwhile, by Ramsey's Theorem, and the assumption $n\ge k^{100t}$, we have that every set of $n^{1/4}$ vertices either has a $k$-clique or independent set of size $t$. Thus, we can find an independent set $I\subset S$ of size $s\le t$ (since $G$ is assumed to be $K_k$-free). Then, $|N(I)|\ge \sqrt{n}$, so we can find an independent set $J\subset N(I)$ of size $t$. But then $G[I\cup J] \cong K_{s,t}$, contradicting the assumption that $G$ was $K_{s,t}$-free.
    \end{proof}
    We can also deduce the following handy result.
\begin{lemma}\label{too dense theta}
    Fix $t\ge s\ge 2$, and let $k$ be sufficiently large. 

    Suppose $G$ is an $n$-vertex graph with $d(G)\ge n^{1-1/(100s)}$ and $n\ge k^{100t}$. Then $G$ must contain either a $K_k$ or an induced $K_{s,t}$.  
\end{lemma}
\noindent The above follows from the second case in our proof of Corollary~\ref{supersat}, we omit the details.

A classical result of Bollob\'as and Thomason~\cite{bollobas} and independently shown by K\'omlos and Szemer\'edi~\cite{komlos} gives the correct bounds for the extremal numbers of a subdivision of a complete graph. 
\begin{theorem}\label{BTsubd}
    Let $G$ be a graph with average degree at least $100k^2$ then $G$ contains a subdivision of a complete graph on $k$ vertices. 
\end{theorem}

We shall use a recent result of Gil Fern\'andez, Hyde, Liu, Pikhurko, and Wu~\cite{QuantSubbalan} which is a quantitative strengthening of a breakthrough result of Liu and Montgomery~\cite{LiuMont}.  
\begin{theorem}\label{balanced subdivision}
    Let $G$ be a graph with average degree $d$. Then, $G$ contains a balanced subdivision of $K_h$, where $h = \Omega(\sqrt{d})$. 
\end{theorem}

\hide{\section{Outline}
\subsection{Structural assumptions}

\subsection{Rough strategy}

Suppose we are given some graph $G$ with average degree at least $d$. To prove Theorems~\ref{thm: main1} and \ref{thm:main2}, we may assume that $G$ that has certain structural properties (namely, it belongs to a certain hereditary family $\FF$), and wish to deduce that $G$ must have a $K_{s,s}$-subgraph for an appropriately large value of $s$. The story for Theorem~\ref{thm: main} is similar, only now we want to find one of two large things. Each of our arguments breaks into two parts. 

In the first phase, we make no additional assumptions about $G$, and iteratively ``clean'' it. We initialize with $G_0 := G$, and pass to induced subgraphs 
\[G_0\supset G_1 \supset \dots G_\tau\]where $d(G_{t+1})\ge d(G_t)^{\Omega(1)}$ for $t<\tau$, and $\tau$ is some stopping time bounded by $3$.

Next, we have a ``clean'' graph $G^*:= G_\tau$, which falls into one of several structured classes. Here, we finally use some assumptions about $G$. Namely, we can now assume that $G^*$ does not contain a $K_{s,s}$-subgraph (or, a $K_k$, and potentially that $G^*$ does not contain some speicifc subgraph (

\subsection{Some details about cleaning}
\zh{not finished yet, feel free to give an attempt. or I can try later.}}

\section{Auxiliary cleaning lemmas }

\subsection{Key dichotomy}
We require a lemma from \cite{JS}. We recreate their proof for completeness.

\begin{definition} We say a bipartite graph $\Gamma = (A,B,E)$ is {\em $L$-almost-biregular} if: $d_\Gamma(a)\le L|E|/|A|$ for all $a\in A$,  and $d_\Gamma(b)\le L|E|/|B|$ for all $b\in B$.
\end{definition}
\begin{lemma}[{\cite[Lemma~3.7]{JS}}]\label{biregular to regular} Let $\Gamma = (A,B,E)$ be $L$-almost-biregular. Then $\Gamma$ has an induced subgraph $\Gamma'$ with with $d(\Gamma')\ge d(\Gamma)/4$ and $\Delta(\Gamma')\le 24Ld(\Gamma')$.
\end{lemma}
\begin{proof}
    If $E = \emptyset$, there is nothing to prove (we may take $\Gamma' = \Gamma)$. Supposing otherwise, we now have that $L\frac{|E|}{|A|},L\frac{|E|}{|B|}\ge 1$.
    
    We may assume $|A|\le |B|$. Let $p = |A|/|B|$. We take a random subset $B'\subset B$ by adding each $b\in B$ independently with probability $p$. We take $A'\subset A$ to be the set of $a\in A$ where $|N_\Gamma(v)\cap B'|\le 1+2p(d_\Gamma(v)-1)$. We shall take $\Gamma' = \Gamma[A',B']$, and show this works with positive probability.

    By construction, we have $\Delta(\Gamma')\le 1+2L\frac{|E|}{|B|}\le 3L\frac{|E|}{|B|}$. Indeed, each vertex $a\in A'$ has degree at most $1+2pL|E|/|A|\le 3L |E|/|B|$, and each vertex $b\in B'$ has degree at most $d_\Gamma(b)\le L|E|/|B|$.

    Consider any $e= ab\in E$, and let $U = (N_\Gamma(a)\setminus \{b\}) \cap B'$. Applying Markov's inequality, we see that \[\Bbb{P}(a\in A'|b\in B') = \Bbb{P}(|U| \le 2\E[|U|])>   1/2.\] It follows that 
    \[\E[e(\Gamma')] = \sum_{e=ab\in E} \Bbb{P}(b\in B')\Bbb{P}(a\in A'|b\in B') > p|E|/2.\]

    Also, it is obvious that $\E[|A'|+|B'|] \le |A|+\E[|B'|]=2|A|$. 
    
    Thus, we have 
    \begin{align*}
        \E[4e(\Gamma') - \frac{|E|}{|B|}(|A'|+|B'|)]
        &> 4(p|E|/2) - \frac{|E|}{|B|}2|A| =0.\\
    \end{align*}Consequently, we can choose $A',B'$ such that the LHS is positive, and so $\Gamma'$ is non-empty with
    \[
    d(\Gamma')\ge 2(\frac{|E|}{4|B|}) \ge d(\Gamma)/4,
    \] 
    and, as $\Delta(\Gamma')\le 3L|E|/|B|$,
    \[
    12Le(\Gamma')\ge \Delta(\Gamma')(|A'|+|B'|) 
    \]
    which implies that $\Delta(\Gamma')\le 24L d(\Gamma')$ 
    as desired.
\end{proof}

Lemma~\ref{biregular to regular} allows us to `bootstrap' almost-regular graphs as follows.
\begin{lemma}\label{bootstrap almost reg}
    Let $L\ge 2$ be a constant and let $G$ be a $n$-vertex graph with average degree $d(G)\ge d$. Furthermore, suppose that $\Delta(G)\le Ld$. Then $G$ has an induced subgraph $H\subset G$ where $\Delta(H)= O(d(H)\cdot \log^2 (L))$ and $d(H)=\Omega(\frac{d}{\log^2(L)})$.
\end{lemma}
\begin{remark}
    We may and will assume $L$ is large enough so that $2\log(L) \ge \log(L)+\log\log(L)+7$ (recall that all logarithms are in base $2$).
    This is simply a straightforward modification of an argument that appeared in the proof of \cite[Lemma~3.4]{c4free}. We need a slightly broader range of parameters for the current paper, so we have presented them here. 

    This result will be black-boxed to prove our key ``dichotomy result'' (Lemma~\ref{Lem:dichotomy}), which is a minor generalization of \cite[Lemma~3.4]{c4free}. Slightly sharper versions of Lemma~\ref{bootstrap almost reg} may be possible by using more involved choices of parameters, in which case we hope the more streamlined presentation will be helpful in future applications.
\end{remark}

\begin{proof}
    First, we may and will assume $\delta(G)\geq d/2$ and that $G$ is $d$-degenerate, otherwise we may pass to a subgraph with higher average degree. Write $V:= V(G)$.

    We now split the vertices in $V$ into $V_i=\{x\in V: 2^{i-2}d\leq d(x)<2^{i-1}d \}$, for $i\in [\log(L)+1] $. 
    By pigeonhole principle, there is some $i\in [\log(L)+1]$, for which \[\sum_{x\in V_i} d(x)\geq \frac{dn}{\log(L)+1}\geq \frac{dn}{2\log(L)}.\] We fix one such $i$.

    For technical reasons, we shall take a random subset of $V_i$ uniformly at random. Let $F\subset V_i$ be such a random set. 
    Moreover, let $F'\subset F$ be the set of vertices $x\in F$ which send at least $d(x)/2$ edges to $V\setminus F$.
    It is easy to see that $\mathbb{P}[x \in F']\geq \mathbb{P}[x\in F]/2= \frac{1}{4}$. Hence, we may fix an outcome where $|F'| \ge \mathbb{E}[|F'|]\geq |V_i|/4$ and therefore $e(F', V \setminus F')\geq \frac{\sum_{x\in V_i} d(x)}{2\cdot 4}\geq \frac{nd}{16\log(L)}$. 
    
    Now, because $G$ is $d$-degenerate, we can find some $F''\subset F'$ such that $\Delta(G[F'']) \le 4d$ and $|F''|\ge |F'|/2$ (indeed, $G[F'']$ must have at most $d|F'|$ edges, thus $|\{x\in F':|N(x)\cap F'|>4d\}|\le |F'|/2$). We note that \[e(F'',V\setminus F'') \ge \frac{e(F',V\setminus F')}{6}\ge\frac{nd}{100\log(L)}\] since $6|N(x)\cap (V\setminus F)|\ge d(x)+2^{i-1} \ge d(x)+d(x')$ for any $x,x'\in F'$.
    Finally, we can move on. We now shall consider a partition of $V\setminus F''$ according to the degree of these vertices into $F''$. 
    
    Specifically, let $U_j=\{x\in V \setminus F'': \frac{d2^{j-1}}{100\log(L)}\leq d_{F''}(x)<\frac{d2^{j}}{100\log(L)}\}$ for $j=0,\dots, \ell:= \log(L)+\log\log(L)+7$. Note this might not be a partition of all the vertices in $V\setminus F''$, however all but at most $\frac{nd}{200\log(L)}$ edges belong to $G[F'',\bigcup_{j=0}^{\ell}U_j]$. 
    Again by pigeonhole principle there some $j\in \{0,1,\dots,\ell\}$, for which $e(G[F'',U_j])\geq \frac{e(F'',\bigcup_{j=0}^{\ell} U_j)}{2\log(L)}\geq \frac{nd}{200\log^2(L)}$. 
    By $d$-degeneracy, there are at most $|U_j|/2$ vertices in $U_j$ which send more than $4d$ edges inside $U_j$. Jettison those vertices and denote the leftover vertices by $U_j'$. This leaves us with $e(G[F'',U_j']) \ge e(G[F',U_j])/3 $.
    
    Now, observe that $G[F'',U'_j]$ is a $M$-almost-regular bipartite graph where $M=1200\log(L)$ since every vertex in $U'_j$ sends roughly (up to a factor of two) the same number of edges to $F''$ and the maximum degree of vertices in $F''$ is at most $2^{i+1}d$ and $e(G[F'',U'_j])\geq \frac{2^{i}d|F''|}{200 \cdot 3\log(L)}$. We may now apply Lemma~\ref{biregular to regular} to $G[F',U'_j]$ with $L:=M$, to find an induced subgraph $H\subset G[F'',U'_j]$ with $d(H)\geq d(G[F',U'_j])/4$ and $\Delta(H)=O(\log^2(L)d(H))$ and so $H$ is the required induced subgraph. There might be some vertices in $H$ which send $4d$ edges within $V(H)\cap U'_j$ or within $V(H)\cap F'$ but this is fine as $d(H)=\Omega(\frac{d}{\log(L)^2})$. 
\end{proof}

The proof of our main result shall be split into two cases. Either we can find an induced subgraph of $G$ which is almost regular and which still has many edges or we can pass to an induced bipartite subgraph $H$ consisting of two parts $A$ and $B$ where $|A|\ggg |B|$ also preserving many edges. This dichotomy is shown in the next lemma.

\begin{lemma}\label{Lem:dichotomy}
    Let $L, d\geq 16$ be positive integers. 

    Consider an $n$-vertex graph $G$ with average degree $d(G)\ge d$, which is $d$-degenerate. Then, one of the following holds. 
    \begin{enumerate}
        \item There is a partition of $V(G)=A\cup B$ where $|A|\geq \frac{L|B|}{2}$ and $e(G[A,B])\geq \frac{nd}{8}$; 
        \item There is an induced subgraph $H\subset G$ where $\Delta(H)= O(d(G[H])\cdot \log^2(L))$ and $d(G[H])=\Omega(\frac{d}{\log^2(L)})$.
    \end{enumerate}
    
\end{lemma}
\begin{proof}
    Let $A_{\text{heavy}}=\{x\in V(G)| d(x)\geq Ld\}$ and\footnote{Perhaps a more apt term is `not-too-abnormally-heavy', rather than `light', but we opt for the latter term for the sake of brevity.} $A_{\text{light}}=V(G)\setminus A_{\text{heavy}}$. Note that by assumption on the degeneracy, we must have that $|A_{\text{heavy}}|\leq 2n/L$. Now first suppose that $e(G[A_{\text{light}},A_{\text{heavy}}]) \geq nd/8$. Then, taking $A:=A_{\text{light}}, B:= A_{\text{heavy}}$ forms the required partition.

    If not, then since $G$ is $d$-degenerate and $|A_{\text{heavy}}|\leq 2n/L$, we have that $e(G[A_{\text{heavy}}])\leq 2nd/L\leq nd/8$. 
    Hence, \[e(G[A_{\text{light}}])= e(G)-e(G[A_{\text{light}},A_{\text{heavy}}])-e(G[A_{\text{heavy}}])\geq n{d}(G)/2-nd/8-nd/8\geq nd/4.\] We are now done by applying Lemma~\ref{bootstrap almost reg} to $G[A_{\text{light}}]$ with $d:= d/2,L:= 2L$.    
\end{proof}

\subsection{The subdichotomy}

Now, in the case where we are almost-regular we show how to further clean our graph. 

\begin{lemma}\label{deletion}
    Fix a graph $F$ and $\eps,\delta>0$ with $\eps<\delta/2$. Suppose that $d$ is sufficiently large with respect to $\eps,\delta$.

    Let $G$ be an $n$-vertex graph with $d(G) \ge d$ and $\Delta(G)\le d^{1+\eps}$. Furthermore, suppose that $G$ has less than $nd^{|V(F)|-1-\delta}$ copies of $F$ (not necessarily induced). Then $G$ contains an induced subgraph $G'$ which does not contain $F$ as a subgraph, such that $d(G')\ge d^{\delta/(10|V(F)|)}$.
\end{lemma}
\begin{proof}
    Let $\FF$ denote the set of subgraphs $F'\subset G$ with $F'\cong F$. Also for convenience write $v_0:= |V(F)|$.
    
    Let $\mathcal{S}$ be the set of pairs $(e,F')$ such that $F'\in \mathcal{F}$ and $e\in E(G)$ is an edge with exactly one vertex in $V(F')$. Clearly, we have $|\mathcal{S}|\le v_0d^{1+\eps}|\mathcal{C}|<v_0nd^{v_0-\delta/2}$.

    Now set $\delta' := \frac{\delta}{5v_0}$, and consider a random set $V'\subset V(G)$ where each vertex is included with probability $p := d^{\delta'-1}$ (independently). We let $\FF'$ be the set of $F\in \FF$ where $V(F)\subset V'$. 
        
    Finally we consider \[ X:= \E[e(G[V'])], \quad Y:=\E[|\FF'|],\quad Z:= \E[\#\{(e,F')\in \mathcal{S}:e\in E(G[V'])\text{ and } F'\in \FF'\}] .\] Let $G'$ be the induced subgraph on vertex set $V' \setminus \bigcup_{F\in \FF'}V(F)$. It is clear that $e(G')\ge X-\binom{v_0}{2}Y-Z$, and that $G'$ will not have $F$ as a subgraph.

    Now, we simply observe that
    \[\E[X] \ge n(d/2)p^2= \frac{1}{2}d^{\delta'}(np),\quad \E[Y]=|\FF|p^{v_0}\le d^{v_0\delta'-\delta}(np) \quad {\text{ and }} \quad \E[Z]= |\mathcal{S}|p^{v_0+1}\le v_0d^{(v_0+1)\delta'-\delta/2}(np) .\]Thus (noting that $(v_0+1)\delta'<\delta/3$, and assuming $d$ is sufficiently large), $Y,Z$ become lower order terms and we get
    \[\E[X-\binom{v_0}{2}Y-Z] \ge d^{\delta'/2}(np)= d^{\delta'/2}\E[|V'|],\]whence there is some outcome where $e(G')\ge d^{\delta'/2}|V(G')|$, giving us our desired subgraph.
\end{proof}

\section{Using the subdichotomy}

In this short section, we collect some useful corollaries of our cleaning lemmas. These will allow us to pass to a favorable outcome if the edges are noticeably concentrated in some local area of $G$.

\begin{lemma}\label{Lem: denseorC_4}
Fix $\varepsilon\in (0,1/10)$ and assume $d$ is sufficiently large with respect to $\varepsilon$. 

Let $G$ be an $n$-vertex graph with $d(G)\geq d$ and $\Delta(G)\leq d^{1+\varepsilon}$. Then, we can either find an induced subgraph $G'\subset G$ on $n'$ vertices where $n'\geq d^{1-3/2\eps}$ and with $d(G')\geq (n')^{1-5\eps}$, or an induced subgraph $G''\subset G$ with $d(G'')\geq d^{\varepsilon/20}$ which does not have any $C_4$.
    
\end{lemma}

\begin{proof}
    Let $\CC_4$ be the set of $4$-cycles $C\subset G$. If $|\CC_4|\le nd^{3-2\eps}$, then we can find $G''$ by applying Lemma~\ref{deletion} (with $F:= C_4$).

    Otherwise, there must be some vertex $x\in V(G)$ which belongs to at least $ d^{3-2\eps}$ cycles of length $4$.
    Next, since $|N(x)| \le d^{1+\eps}$, there must be some $y\in N(x)$ such that $\{x,y\}$ is an edge which belongs to at least $d^{2-3\eps}$ cycles of length $4$. But the number of such cycles exactly counts the number of choices of $x'\in N(x)\setminus \{y\}, y'\in N(y)\setminus \{x\}$ where $\{x',y'\}\in E(G)$. Thus taking $G'= G[N(x)\cup N(y)]$ gives a graph with $\ge d^{2-3\eps}/2$ edges and $m:=|V(G')|\le 2d^{1+\eps}$ vertices. This implies that $d(G')\ge d^{1-4\eps}/4\ge m^{1-5\eps}$ (using that $\eps<1/10$ and $d$ is sufficiently large). 
\end{proof}

\begin{lemma}\label{independent set or dense}
    Fix $\eps\in (0,1/10)$, and suppose $n$ is sufficiently large with respect to $\varepsilon$. 

    Let $G$ be an $n$-vertex graph, without an independent set of size $\sqrt{n}$. Then, $G$ either has an induced subgraph $G'$ on $n'\ge (n')^{1/4}$ vertices where $d(G')\ge (n')^{1-5\eps}$, or an induced subgraph $G''\subset G$ with $d(G'')\ge n^{\eps/100}$ which does not contain any $C_4$.
\end{lemma}
\begin{proof}
    By T\'uran's Theorem, we may assume $d_0:=d(G)\ge \sqrt{n}/2$, otherwise we have an independent set of size $\sqrt{n}$. Applying the dichotomy from Lemma~\ref{Lem:dichotomy} with $L = 2^{d_0^{\eps/12}}\ge 2^{n^{\eps/25}}$ and assuming $n$ is sufficiently large, we have that $L/2>n$, whence the first outcome cannot happen, which means we can pass to an induced subgraph $G^*$ of $G$ with $d(G^*)\ge \Omega(d_0/d_0^{\varepsilon/12})\ge  \Omega(n^{1/2-\eps/6})$ and $\Delta(G^*)\le O(d(G^*)n^{\eps/6})$.

    Again assuming $n$ is sufficiently large, we have that $\Delta(G^*) \le d(G^*)^{1+\eps}$, so we can now apply Lemma~\ref{Lem: denseorC_4} to get the result (taking $d:= d(G^*)\ge n^{1/3}$). 
\end{proof}

\hide{\zh{I don't know where this should go now.}Obviously this implies that $\chi(G)=C(\ell, t)k^{O(t)} $. 
In the concluding remarks we shall explain how it follows easily from a result of Chudnovski, Seymour, Scott and Spirkl an extension of Theorem~\ref{thm: main}, namely that families of graphs avoiding $k$ vertex disjoint and pairwise non-adjacent $t$-theta graphs are polynomially $\chi$-bounded.}

\section{Finding and using subdivisions}

\subsection{A reduction}

\begin{definition}
    Given a multihypergraph $\HH = (V,E)$, we define the $1$-subdivision of $\HH$ to be the bipartite graph $\Gamma$ with bipartition $A,B$, having bijections $\phi_A:A\to E,\phi_B:B\to V$, so that
    \[N_\Gamma(a) = \{b\in B: \phi_B(b)\in \phi_A(a)\}\]for each $a\in A$.
    
\end{definition}

By definition, it is obvious that the following property holds.

\begin{observation}\label{subdivision subgraphs are induced}
    Let $G$ be the $1$-subdivision of a multihypergraph $\mathcal{H}$. Suppose $\mathcal{H}'$ is a subhypergraph of $\mathcal{H}$ (obtained by removing some edges and vertices), then the $1$-subdivision of $\mathcal{H}'$ is an induced subgraph of $G$.  
\end{observation}
\noindent This is quite helpful, since it allows us to reduce problems about finding induced subgraphs in $G$ into problems about finding (not necessarily induced) subgraphs inside the auxiliary graph $H$ (which are understood much better).

We show now how to find a $1$-subdivision of a \textit{simple} graph with many edges from a $1$-subdivision of an $s$-uniform hypergraph of high enough average degree, provided the hyperedges do not clump too much.  
\begin{proposition}\label{drop uniformity}
    Let $s,t\geq 2$ and $H$ be a $1$-subdivision of a $s$-uniform multihypergraph $\HH$, with $d(\HH) \ge d$. Suppose $H$ is $K_{s,t}$-free, then $H$ contains an induced subgraph $H'$ which is a $1$-subdivision of a simple graph $G$ with average degree $\Omega_{s,t}(d^{1/(s-1)})$.
\end{proposition}
\begin{proof}
    We will first remove some hyperedges from $\HH$, to obtain a ``cleaner'' hypergraph $\HH^*$. This corresponds to deleting vertices from $H$.
    
    First, since $H$ is $K_{s,t}$-free, no hyperedge $e\in E(\HH)$ can have multiplicity greater than $ t-1$. Thus, by removing some hyperedges, we can obtain a simple hypergraph $\HH^*\subset \HH$ with $d(\HH^*)\ge d/t$. 
    
    Next, by randomly partitioning the vertices of $\HH^*$ into $s$ parts, we can find parts $V_1,\dots,V_s$, so that the $s$-partite subhypergraph $\HH^{**} := \HH^*[V_1,\dots,V_s]$ has $d(\HH^{**}) \ge d(\HH^*)s!/s^s= \Omega_{s,t}(d)$.

    Finally, pass to a choice of $V'$ where $\HH^{***}:= \HH^{**}[V']$ has maximal average degree, and write $V_i':= V'\cap V_i$ for $i\in [s]$. We have that $d_{\HH^{***}}(v)\ge d':= \frac{d(\HH^{**})}{s}= \Omega_{s,t}(d)$ for each $v\in V'$.

    We may and will assume $|V_s'|\ge |V_i'|$, for all $i\in [s-1]$. For each $v\in V_s'$ and $i\in [s-1]$, write \[N_i(v) := V_i' \cap \bigcup_{e\in E(\HH^{***}):v\in e}e.\]Clearly, we have that $ \prod_{i=1}^{s-1}|N_i(v)|\ge d_{\HH^{***}}(v)\ge d'$. Thus for each $v\in V_s'$, there is some choice of $i=i_v$ so that $|N_i(v)|\ge d'^{1/(s-1)}$. By pigeonhole, we can find some choice of $i\in [s-1]$ so that $\sum_{v\in V_s'}|N_i(v)| \ge \frac{d'^{1/(s-1)}}{s-1}|V_s'| = \Omega_{s,t}(d^{1/(s-1)})|V_s'\cup V_i'|$ (recall $|V_s'|\ge |V_i'|$). Consider the graph $G$ with vertex set $V(G)\coloneqq V'_s\cup V'_i$ where $(x_s,x_i)\in E(G)$, if $x_s\in V'_s $ and $x_i\in V'_i$ and $\{x_s,x_i\}\subset e$, for some $e\in \mathcal{H}^{***}$. 
    By the above $d(G)=\Omega_{s,t}(d^{1/(s-1)})$. Finally, note that the $1$-subdivison of $G$ is an induced subgraph of $H$. 
\end{proof}
Combining Proposition~\ref{drop uniformity} with Theorem~\ref{balanced subdivision} and Observation~\ref{subdivision subgraphs are induced}, we immediately get the following (since the $1$-subdivision of a balanced subdivision of $K_h$ is still a balanced subdivision of $K_h$).
\begin{corollary}\label{subdivision reduction}
    Fix $t\ge s\ge 2$. Let $G$ be a $K_{s,t}$-free graph and suppose it contains an induced $1$-subdivision of an $s$-uniform multihypergraph $\HH$. Then, $G$ contains an induced proper balanced subdivision of $K_h$, where $h = \Omega_{s,t}(d(\HH)^{1/2(s-1)})$.
\end{corollary}
\begin{proof}
    Applying Proposition~\ref{drop uniformity}, $G$ contains an induced subgraph $G'$, which is the $1$-subdivision of some simple graph $H$ with $d(H) = \Omega_{s,t}(d(\HH)^{1/(s-1)})$. Then by Theorem~\ref{balanced subdivision}, $H$ contains a balanced subdivision of $K_h$ for some $h= \Omega(d(H)^{1/2})$ as a (not necessarily induced) subgraph. Finally, Observation~\ref{subdivision subgraphs are induced} will tell us that $G'$ contains the $1$-subdivision of this balanced subdivision of $K_h$, as an induced subgraph. This gives us our desired induced subgraph $H'$. 
\end{proof}

\subsection{Finding subdivisions of multihypergraphs}

\begin{lemma}\label{Lem:unbalanced}
   Let $k,s,d\ge 1$ be integers so that the following estimates hold: 
   \[ d\ge (2(10+s))^s, \quad \text{and} \quad d/20>\binom{s+k}{s} .\]
   Let $G$ be a $d$-degenerate graph with a bipartition $V(G)=A \cup B$ with $e(G[A,B])\geq \frac{d|A|}{10}$ and $|A|> 40d^{s+2}|B|$. If $\omega(G)\leq k $, then $G$ contain an induced $1$-subdivision of a $s$-uniform hypergraph $\HH$ with $d(\HH)\ge d$. 
\end{lemma}

\begin{proof}
    First, since $G$ is $d$-degenerate, one can prove the following claim.
    \begin{claim}\label{make regular}
        There exists some $A'\subset A$ with $|A'|\ge |A|/40d$, so that $G[A'] = \emptyset$ and $d_B(x) \in [d/20,10d]$ for each $x\in A'$.
    \end{claim}
    \noindent We shall delay this proof for the time being. Instead, let us fix such a choice of $A'$ and focus on $G[A'\cup B]$. We have that $|A'|> d^{s+1}|B|$.

    By $d$-degeneracy we may order the vertices in $B$, say $b_1,\ldots b_{|B|}$, such that for all $i\in [|B|]$, $b_i$ sends at most $d$ edges to $b_1,b_2,\ldots, b_{i-1}$. We now take a random subset of $B$ where each vertex is chosen with probability $p=\frac{1}{2(10+s)d}$ independently. Let $B_p$ be the random subset. 
    Let $x\in A'$ and $X(x)$ be a random variable which is $1$ if $x$ has exactly $s$ non-neighbours $b_1,b_2,\dots,b_s \in B_p$ such that each of them send no edges to its left in the ordering, and $0$ otherwise. 
    
    We will now compute $\mathbb{E}[X(x)]$. Let $N_B(x)=\{b_{i_1},\ldots, b_{i_{\ell}}\}$, for some $\ell \in [d/20,10d]$, with $i_1<i_2<\ldots <i_{\ell}$. 
    Note that since $d/20> {s+k \choose k}$, Ramsey's Theorem guarantees an independent set of size $s$ inside $N_B(x)$. Thus,
    \[
    \mathbb{E}[X(x)]=\mathbb{P}[X(x)=1]\geq p^s (1-p)^{sd+\ell}> p^s(1-(10+s)dp)\ge d^{-{(s+1)}}/2.
    \]
    Finally, \[
    \mathbb{E} \big [\sum_{x\in A'}X(x)-d|B_p|\big ]\ge \frac{|A'|}{2d^{s+1}}-\frac{1}{2(10+s)}|B|>0
    \]
    Fix a choice for which the above is positive. We have thus constructed a subset $A''\subset A'$ where every vertex in $A''$ sends exactly a pair of edges to $B_p$,
    $G[B_p]=\emptyset$ and $|A''|\geq d|B_p|$. Since each $a\in A''$ has exactly $s$ neighbors in $B_p$, we see that $G' :=G[A''\cup B_p]$ is the 1-subdivision of an $s$-uniform multihypergraph $\HH$, moreover we have that $d(\HH) = s|A''|/|B_p|\ge d$ as desired. 
 
    \begin{proof}[Proof of Claim~\ref{make regular}]
        First, write $A_- := \{x\in A: d_B(x) <d/20\}$. Noting that $e(G[A\setminus A_-,B])\ge \frac{d|A|}{20}$, we must have $|A\setminus A_-|\ge |A|/10$, otherwise we will have $d(G[A\setminus A_-,B])>d$, contradicting that $G$ is $d$-degenerate.

        Next, write $A_+:= \{x\in A: d_B(x)>10d\}$. We must have $|A_+|\le |B|$, otherwise $d(G[A_+,B])>d$ (again contradicting $d$-degeneracy).

        So now write $A_0 := \{x\in A: d_B(x)\in [d/20,10d]\}$. We have that $|A_0|\ge |A|/10-|B|\ge |A|/20$.

        Finally, since $G$ is $d$-degenerate, we can partition $A_0$ into $d+1\le 2d$ independent sets. Take $A'$ to be the largest of these sets which will have size at least $ |A|/40 d$.
    \end{proof}
    
\end{proof}

On the other hand, if we are almost-regular, then we have the following result. 

\begin{lemma}\label{almost regular}
    Fix $t\ge s\ge 2$, and set $\eps_0 = 1/(50000s^2)$. The following holds for all sufficiently large $k$ with respect to $s$ and $t$.

    Let $G_0$ be a graph with $\omega(G_0)\le k$ and $d(G_0)= d_0\ge k^{1000t}$. Suppose furthermore that $\Delta(G_0) \le d_0^{1+\eps}$ and $G_0$ is $K_{s,t}$-free. Then $G_0$ contains an induced $1$-subdivision of an $s$-uniform multihypergraph $\HH$ with $d(\HH)\ge \Omega_{s,t}(d_0^{1/(40000s)})$.
\end{lemma}
\begin{remark}
    The exponents are sharp (up to the constants `$1000$' and `$1/40000$') for a few reasons. By considering $G\sim G(2k^{t/1000},1-1/k^{1/2})$ for $k\gg t$, we can get graphs with average degree $k^{t/1000}$ that lack independent sets of size $t$ (so it is $K_{s,t}$-free and has no $1$-subdivision on $>2t$ vertices), and no cliques of size $k$. Meanwhile, if we consider $G\sim G(n,d/n)$, then some calculations involving Chernoff bounds \hide{(cf. https://arxiv.org/pdf/2207.02170.pdf, Theorem 1.1)}tell us there exist $G$ with average degree $>d/2$, no $4$-cycles and no $1$-subdivisions of $s$-uniform multigraphs $\HH$ with $d(\HH) > d^{1/(s-1)+\eps}$ (for fixed $s\ge 2,\eps>0$, and all large $d$). 
\end{remark}
\noindent Combined with Corollary~\ref{subdivision reduction}, Lemma~\ref{almost regular} would allow us to find a balanced subdivision of $K_{h'}$ for some $h' = \Omega_{s,t}(d_0^{1/(1000000s^2)})$ inside $G_0$. Thus this will not give the sharp exponent of $\Omega(1/s)$ in the second bullet from Theorem~\ref{thm: main}. 

Instead, to get the $\Omega(1/s)$ bound, it will suffice to establish a polynomial bound when $s=t=2$, and then bootstrap that. But we saw no reason not to prove the result for general $s,t$ here.

\begin{proof}
    Write $d:= d_0^{1/10},\eps:= 20\eps_0$. We shall use the following cleaning result.
    \begin{claim}\label{almost regular, cleanup}
        Let $G_0$ be as above. Then we can find an induced subgraph $G\subset G_0$ with $d(G)\ge d/10, \Delta(G')\le d^{1+\eps}$, where $|N_{G'}(x)\cap N_{G'}(y)|\le 3d^{1-1/(1000s)}$ for all $y\in N_{G'}(x)$.  
    \end{claim}
    \noindent We shall prove Claim~\ref{almost regular, cleanup} by sampling vertices uniformly at random with probability $d_0^{-9/10}$ and using the fact that $G_0$ is $K_k$-free and $K_{s,t}$-free. We delay the details until the end of the proof, and now show how to find the 1-subdivision of some appropriate hypergraph $\HH$.

    Without loss generality, we may assume that $\delta(G)\ge d/20$, else we can pass to an induced subgraph of $G$ with higher average degree.

    Let $\eta=1/2000$ (chosen so that $\eta/(2s)>(2s+3)\eps$ and $\eta<1/1000$), and write $\delta := \eta/s$. We assume $k$ is sufficiently large so that $s d^{-1/1000} < 1/2$.

    Take $p := d^{-1-\delta}$, and let $B_0$ be random subset of $V(G)$ where each vertex is included with probability $p$. We then define $A_0$ to be the set of vertices with exactly $s$ neighbors inside $B_0$. Then define $B$ to be the set of vertices $b\in B_0$ with $|N(b)\cap B_0| = 0$. Finally, we define $A$ to be the set of vertices in $a\in A_0$ where $N(a)\cap B = N(a) \cap B_0$ (in particular, this implies $a\not\in B_0$, since $a$ is adjacent to vertices in $B$).

    We will now analyze the expected number of vertices inside $A$ and $B$, and the number of edges within $G[A]$. Set $n:= |V(G)|$.

    For $B$, it suffices to note that
    \[\P(x\in B) \le \P(x\in B_0) =p\]for $x\in V(G)$.

    Understanding $A$ and $e(G[A])$ is more sensitive. For $x\in V(G)$, set $P_x := (d(x)p)^s$. Assuming $k$ is sufficiently large, we shall establish for every $x\in V(G)$ and $y\in N(y)$ that
    \begin{equation}
        \P(x\in A) = \Theta_s(P_x)\label{xprob} 
    \end{equation}
    \begin{equation}
        \P(x\in A\text{ and } y\in A)\ \le O_s(P_xP_y)\label{xyprob}. 
    \end{equation}
    \noindent From here, we get that $\E[|A|] \ge n \Omega_s(d^{-\eta})$ (since $G$ having minimum degree $\ge d/20$ implies $P_x\ge (d^{-\delta}/20)^s$). Meanwhile, $\E[e(G[A])] \le e(G)\max_{x\neq y}\{P_xP_y\} \le O_s(nd^{1-2\eta+(2s+1)\eps})$.

    So now, let $A'$ be a random subset of $A$ where each vertex is kept with probability $p':= d^{-(1-\eta+(2s+2)\eps)} = O_s(d^{-\eps}\frac{\E[Y]}{\E[X]}) $. We have that $\E[|A'|-e(G[A'])]= p'\E[|A|]-p'^2\E[e(G[A])]>(p'/2)\E[|A|] > nd^{-1-(2s+3)\eps}$ for all large $d$. Thus, we can find some outcome where
    \[\E[|A'|-e(G[A'])-d^{-\delta/2}|B|] >0.\] By deleting a vertex from each edge in $A'$, we obtain independent sets $A'',B$ with $|A''|>d^{-\delta/2}|B|$ and $d_B(x) = s$ for each $x\in A''$. This gives our desired subdivision.

    We now prove Equations~\ref{xprob} and \ref{xyprob}. Here, we shall frequently use the inequalities $sd(x)p<1/2$ and $d(x)p< d^{\eps-\delta}d^{-1/(1000s)}$ for any $x\in V(G)$, which are guaranteed by our assumptions. So fix some $x$, and consider $N(x)$.
    
    Let $\II$ be the collection of independent sets $I\subset N(x)$ of size $s$. By Corollary~\ref{supersat}, we get that $|\II|= \Omega_{s}(d(x)^{s})$ since $d(x)\ge d/20\ge k^{100t}$. And obviously $|\II|\le \binom{d(x)}{s}\le d(x)^s$. Thus, we have that \[\P(x\in A_0)= \sum_{I\in \II}\P(B_0\cap N(x) = I) = |\II|p^{s}(1-p)^{d(x)-s}\ge \Theta_s(d(x)^{s})(p^{s}/2)= \Theta_s(P_x).\]
    Furthermore,
    for $I\in \II$ we have\[\P(x\in A|B_0\cap N(x) = I) =  \P(B_0 \cap \bigcup_{y\in I} N(y) = \emptyset|B_0\cap N(x) = I) = (1-p)^{|\bigcup_{y\in I}N(y)|}\ge 1-sd^{1+\eps}p \ge 1/2.\]So we see $\P(x\in A) \ge \P(x\in A_0)/2$, whence $\P(x\in A) = \Theta_s(P_x)$ (since $A\subset A_0$). This establishes Equation~\ref{xprob}.

    Finally, we bound $\P(y\in A_0| x\in A_0)$ for $y\in N_G(x)$. Let $\EE_x$ be the event that $|B_0\cap N(x)| = s$. Again by Corollary~\ref{supersat}, we see that $\P(x\in A_0|\EE_x) = \Omega_s(1)$. So now, for any vertex $y\in V(G)$, we have that
    \[\P(y\in A_0 |x\in A_0) \le \frac{\P(y\in A_0| \EE_x)}{\P(x\in A_0|\EE_x)}\]by Bayes' theorem, and
    \begin{align*}
        \P(y\in A_0|\EE_x)&= \sum_{i=0}^{s} \P(|B_0\cap N(x)\cap N(y)| =i|\EE_x)\P(|B_0\cap (N(y)\setminus N(x))| = s-i)\\
        &\le \sum_{i=0}^{s} O_s\left(\frac{|N(x)\cap N(y)|}{d(x)}\right)^i (d(y)p)^{s-i}\\
        &\le O_s( \max\left\{\frac{|N(x)\cap N(y)|}{d},d(y)p\right\}^{s}).
    \end{align*}Thus, for $y\in N(x)$, we have $\P(y\in A_0|x\in A_0)\le O_s(d^{-1/1000}+P_y) = O_s(P_y)$ (since $P_y =\Omega_s(d^{-\eta})$). This establishes Equation~\ref{xyprob}.

    \begin{proof}[Proof of Claim~\ref{almost regular, cleanup}]
    
    For $x\in V(G)$, write $Y_{\text{heavy}}(x)\subset N(x)$ to be the set of $y\in N(x)$ with $|N(y) \cap N(x)|\ge d_0^{1-1/(1000s)}$. Noting $d_0^{1-1/1000s}\ge |N(x)|^{1-1/100s}$ (recall $|N(x)|\le d_0^{1+\eps_0}$ and $\eps = 1/(5000s^2)$), we must have that $|Y_{\text{heavy},x}|\le \sqrt{|N(x)|}\le d_0^{2/3}$ by Lemma~\ref{too dense theta} (since $G_0$ has not $K_k$ and is $K_{s,t}$-free).

    Now $W$ sample vertices with probability $d_0^{-9/10}$, and let \[Z_1 := \{x\in W: d_W(x) > 1+ 10d_0^{1/10+\eps_0}\}\]
    \[Z_2 := \{x\in W: |N(x)\cap N(y)\cap W|> 1+2d_0^{1/10-1/(1000s)} \textrm{ for some }y\in N(x) \cap W\}.\]
    Now take $W' := W \setminus (Z_1\cup Z_2)$ and $G := G_0[W']$. Deterministically we have $\Delta(G)\le 11d_0^{1/10+\eps_0}$ (which is less than $d_0^{1/10(1+20\eps_0)}$ for large $d_0$) and that $|N_G(x)\cap N_G(y)|\le 3d_0^{1/10(1-1/(1000s))}$ for adjacent $x,y\in W'$. So we just seek an outcome with the right average degree.

    Clearly, we have $\E[|W'|]\le \E[|W|] = nd_0^{-9/10}$. The result will now follow if we can get a good lower bound on $\E[e(G)]$.

    Let $E_{\text{light}}$ be the set of edges $e = \{x,y\}$ where $|N(x)\cap N(y)|< d_0^{1-1/(1000s)}$. It is not hard to see that $|E_{\text{light}}|\ge e(G_0)\ge nd_0/4$ (since each $x$ belongs to exactly $|Y_{\text{heavy},x}|\le d_0^{2/3}$ edges not in $E_{\text{light}}$). For every $e= \{x,y\}\in E_{\text{light}}$, we claim that
    \begin{equation}
        \P(\{x,y\}\subset W' | \{x,y\}\subset W)\ge 1/2.\label{edge retention}
    \end{equation}This will imply that $\E[e(G)] \ge \frac{(d_0^{-9/10})^2}{2}|E_{\text{light}}| > \frac{d_0^{1/10}}{10}\E[|W'|]$, so there is an outcome of $G$ with average degree $\ge d_0^{1/10}/10$, as desired.

    To establish Equation~\ref{edge retention}, we just consider $e=\{x,y\}\in E_{\text{light}}$. Conditioning on $\{x,y\}\subset W$, we will control the probability that $\{x,y\}\cap (Z_1\cup Z_2)$ is non-empty.

    Controlling $Z_1$ is easy. By Markov's inequality (in a similar fashion to in Lemma~\ref{biregular to regular}), we have that $\P(x\in Z_1|\{x,y\}\in W)\le \frac{\E[|W\cap (N_{G_0}(x)\setminus \{y\})|]}{10d_0^{1/10+\eps_0}}<1/10$.

    Meanwhile for $Z_2$, we have that \[\P(x\in Z_2|\{x,y\}\subset W)\le \sum_{y' \in N_{G_0}(x)} \P(y'\in W)\P(|(N(x)\setminus \{y\})\cap N(y')\cap W|>2d_0^{1/10-1/(1000s)}).  \]Since $|Y_{\text{heavy}}(x)|\le d_0^{2/3}$, and the contribution from those vertices in the above sum is at most $d_0^{2/3}d_0^{-9/10}<1/20$. Meanwhile, for any $y'\in N_{G_0}(x)\setminus Y_{\text{heavy}}(x)$, we have that $\P(|(N(x)\setminus\{y\})\cap N(y')\cap W|\ge 2d_0^{1/10-1/(1000s)}) \le \P(\operatorname{Bin}(m,p)\ge 2mp)$ with $m:= d_0^{1-1/(1000s)},p:= d_0^{-9/10}$. By a Chernoff bound, we have that $\P(\operatorname{Bin}(m,p)\ge 2mp)\le \exp(-mp/3)$ for arbitrary $m\ge 1,p\in (0,1)$. So for large $d_0$, we have that $\exp(-d_0^{1/10-1/(1000s)})<d_0^{-(1+\eps)}/20$. Whence, the contribution from the other terms is at most $1/20$.

    So altogether $\P(x\in Z_1\cup Z_2|\{x,y\}\subset W)\le 1/5$. And by symmetry the same should hold for $y$. Thus $\P(\{x,y\}\subset W'|\{x,y\}\subset W) \ge 1-\P(\{x,y\}\cap (Z_1\cup Z_2)|\{x,y\}\subset W)\ge 3/5\ge 1/2$. This completes the proof.
    \end{proof}
\end{proof}

\section{Proof of Theorem~\ref{thm: main}}

\subsection{Avoiding $K_{2,2}$}

We first go through the proof for the case of $s=t=2$ (i.e. $G$ has no induced $K_{2,2}$). This result will allow us to finish off our proof of the general case by passing to some induced $G'$ that contains no $K_{2,2}$. 
\begin{proposition}\label{base case}
    For all sufficiently large $k$, the following holds.
    Let $G$ be a $C_4$-free graph with average degree $d:=d(G)\ge k^{2500}$. Furthermore, suppose $G$ has no $K_k$. Then, $G$ contains an induced (proper) balanced subdivision of $K_h$, where $h\ge d^{1/50000}$. 
\end{proposition}
\begin{proof}
    By our reductions (specifically Corollary~\ref{subdivision reduction}), it suffices to find an induced subgraph $G'\subset G$, which is the $1$-subdivision of some multihypergraph $\mathcal{H}$ with average degree $d(\mathcal{H})>Ch^2$ (here $C$ is some absolute constant).

    We may assume $G$ is $d$-degenerate, otherwise we can pass to some subgraph with higher average degree. We now invoke our dichotomy result (Lemma~\ref{Lem:dichotomy}) with $L:= 80d^4$. 
    
    Suppose the first outcome holds. Then, there exist disjoint sets $A,B$ with $|A|>40d^4|B|$ and $e(G[A,B])\ge e(G)/4$. We may now finish off by invoking  Lemma~\ref{Lem:unbalanced} and obtain the desired $G'$. 

    Otherwise, the second outcome holds. We may then pass to an induced subgraph $G^*$ with $d^*:=d(G^*)\ge \Omega(d/\log(d^4))$ and $\Delta(G^*)=O(d^*\log(d^4))$. Assuming $k$ (and thus $d$) is sufficiently large, we get that $d^*\ge k^{2000}$ and $\Delta(G^*)\le (d^*)^{1+1/200000}$. Finally, we apply Lemma~\ref{almost regular} to find our $G'$ (since we may assume $d^*\ge d^{4/5}$).
\end{proof}

\begin{remark}
    There is actually quite a lot of flexibility with how one chooses $L$. As long as $L<2^{d^{o(1)}}$, in the second outcome of Lemma~\ref{Lem:dichotomy} we are sufficiently close to being regular for the same proof to follow through. 
\end{remark}

\subsection{General case}

We now seek to prove the general case. Here, we need an argument to rule out the case when $G$ is `polynomially dense'. This will be handled by invoking Lemma~\ref{too dense theta}.

\begin{proof}[Proof of Theorem~\ref{thm: main}]
    Write $d := d(G)\ge k^{4000t}$.

    Our goal will be to find an induced subgraph $G'\subset G$ where either:
    \begin{itemize}
        \item $G'$ is the $1$-subdivision of an $s$-uniform multihypergraph $\HH$ with $d(\HH)\ge d/100$;
        \item $G'$ is $C_4$-free with $d(G')\ge d(G)^{\frac{1}{20000s}}$. 
    \end{itemize}\noindent If the first bullet happens, then we are done by Corollary~\ref{subdivision reduction}. And if the second bullet happens, then we are done by Proposition~\ref{base case}.

    Again, we can assume that $G$ is $d$-degenerate, otherwise we could pass to a subgraph of higher average degree. We then apply our dichotomy (namely Lemma~\ref{Lem:dichotomy}) with $L \coloneqq d^{s+3}$.

    Suppose the first case of the dichotomy holds. Then, we can obtain a $G'$ satisfying the first bullet by applying Lemma~\ref{Lem:unbalanced}.

    Meanwhile, if the second case of the dichotomy holds, we have found an induced subgraph $G^*\subset G$ with $d^*:= d(G^*)\ge d^{1/2}$ and $\Delta(G^*)\le (d^*)^{1+1/(500s)}$. Now applying Lemma~\ref{Lem: denseorC_4} to $G^*$ (with $\eps:= 1/500s$), we either get an induced $G'\subset G^*$ satisfying the second bullet or an induced subgraph $G^{**}$ on $n^{**}\ge (d^{*})^{1/2}$ vertices with $d^{**} \ge (n^{**})^{1-1/(100s)}$. But this latter outcome is impossible by Lemma~\ref{too dense theta} (as we are assuming $G$ is $K_{s,t}$-free and has no $K_k$). Thus we see that we must find one of the desired $G'$, completing the proof.
\end{proof}

\section{Theorem~\ref{thm: main1} and Theorem~\ref{thm:main2}}

\subsection{Bounded induced subdivisions}
We proceed as in the proof of Theorem~\ref{thm: main} except that now things are a bit more involved. 

First, we need an analogue of Lemma~\ref{too dense theta},  which will again allow us to ``win'' if we can pass to an induced subgraph $G'$ with $n'\ge d^{\Omega(1)}$ vertices where $d(G')\ge (n')^{1-\eps}$ for an appropriately small choice of $\eps$.

\begin{lemma}\label{one sided erdos hajnal}
    Fix a bipartite graph $H = (A,B,E)$. Write $\eps_H:= \frac{1}{100 \Delta(H)}$ and $C^*_H := 100|A|+10|A|^2+2|B|$. Then the following holds for all sufficiently large $s$ with respect to $|H|$. 

    Let $G$ be a graph with $n\ge s^{C^*_H},d(G)\ge n^{1-\eps_H}$, which has no $K_{s,s}$ (as a subgraph). Then $G$ contains an induced copy of $H$.
\end{lemma}
\begin{remark}
    This result can be viewed as a \textit{one-sided} Erd\H{o}s-Hajnal result (in the same spirit as a result of Fox and Sudakov~\cite[Theorem~1.8]{FoxSud}). Our proof gives better bounds (\cite{FoxSud} required $n>\Omega_H(s^{\Omega(|H|^3|)})$) and has a weaker assumption (they require that $G$ lacks independent sets of size $n^{\Omega(1)}$, which we have replaced by an average degree condition).\hide{\zh{okay? or is this rude?}}
\end{remark}
\begin{proof}
    Write $a:= |A|$, and $\Delta:= \Delta(H)$.

     First, note that we can find $A' \subset V(G)$ with $|A'| = \sqrt{n}$ so that every $x\in A'$ has $d_G(x)\geq n^{1-2\varepsilon_H}$. By Lemma~\ref{DRC}, we must have a set $A^*\subset A'$ where every $\Delta$-subset inside $A^*$ has at least $ n^{9/10}$ common neighbors inside $V(G)$ and $|A^*|\geq n^{1/3}$. Set $B^*:= V(G)\setminus A^*$.
    We need now a simple observation. 

    \begin{observation}\label{Obs1}
        For every set $W\subset V(G)$ with $|W|\ge 2s$, we have that $\#\{x\in V(G): |W\setminus N(x)|\le |W|/2s\} <s$.
    \end{observation}
  \noindent This clearly holds as we do not have any $K_{s,s}$.

    Now pick $x_1,\dots,x_a\in A^*$ independently and uniformly at random. For $S \subset [a]$, let \[Y_S := \{y\in V(G): x_i\in N(y) \iff i\in S\quad \text{for }i\in [a]\}.\]Whenever $|S|\le \Delta$, we will argue that with very high probability, $|Y_S|\ge n^{9/10}/(2s)^{a-|S|}$, for every $S\subset [a]$. To start, observe that $Y_S$ has the same distribution as $Y_{\{1,\dots,|S|\}}$ for every $S\subset [a]$. 
    
    Now fix $0\le \ell \le \Delta$. For $t=1,2,\dots,a$, define
    \[Y^{(t)} =Y_{[\ell]}^{(t)} := \{y \in V(G): y \in N(x_i) \iff i\le \ell \quad\text{for }i\in [t]\}.\]
    By assumption on $A^*$, we have $|Y^{(\ell)}| \ge n^{9/10}$ with probability $1$. Next, for all $t\in \{\ell,\dots, a-1\}$, conditioned on the event $\EE_t$ that $|Y^{(t)}|\ge 2s$, we have that $\P(|Y^{(t+1)}|< |Y^{(t)}|/2s ) < s/|A^*| \le n^{-1/4}$ by Observation~\ref{Obs1}. So iterating through each $\ell\le t<a$, we see that $\P(|Y^{(a)}|\ge n^{9/10}/(2s)^{a-\ell}) \ge (1-n^{-1/4})^{a-\ell}\ge 1-a n^{-1/4}$.
    
    So by a union bound, we get that \[\P\left (|Y_S|< n^{4/5}\quad \text{for some }S\subset [a]\text{ with }|S|\le \Delta\right )\le 2^a a n^{-1/4}\le n^{-1/3}\](here the estimates $n^{4/5}\le n^{9/10}/(2s)^a$ and $2^a n^{-1/4}\le n^{-1/3}$ follow from the assumption that $n\ge s^{100a}$).

    Meanwhile, again by Observation~\ref{Obs1}, we have that $\{x_1,\dots,x_a\}$ is an independent set of size $a$ with probability greater than $ \prod_{i=1}^a \left[\frac{1}{|A^*|}\left(\frac{|A^*|}{(2s)^{i-1}}-s-(i-1)\right)\right]\ge  \frac{1}{(4s)^{a^2}} $ (assuming $|A^*| = n^{1/3} \ge (2s)^a$). 
    
    Thus, since $n\ge s^{10a^2}$, we can find some independent set $I\subset A^*$ of size $a$, so that $|Y_S|\ge n^{4/5}$ for all $S\subset [a]$ with $|S|\le \Delta$. We may now find the required $H$.
    
    Fix some $\iota:A\to [a], \psi:B\to [|B|]$. For each $j\in [|B|]$, set $S_j := \iota(N(\psi(j)))$, and pick some vertex $y_j\in Y_{S_j}$ uniformly at random (and independently of the other $y_j$'s). Similar to before, we have that $\{y_1,\dots,y_{|B|}\}$ is an independent set of size $|B|$ with probability at least $\prod_{j=1}^{|B|} \frac{1}{|S_j|}(\frac{|S_j|}{(2s)^{j-1}}-(|B|-(j-1))s-(j-1)-a)>0$ (assuming $n^{4/5}> (|A|+|B|)(2s)^{|B|}$, which certainly holds for large enough $s$). This yields our copy of $H$.  
\end{proof}

We now require a ``polychotomy'' variant of Lemma~\ref{Lem:unbalanced} in order to handle unbalanced graphs. Here, either something ``abnormal'' happens (one of the three bullets listed below), or a fourth ``expected'' outcome must hold; in each case we shall pass to some induced $G'$ which will lead us to victory.

\begin{proposition}\label{polychotomy Kh}
    Let $h\ge 2$ and suppose $d$ is sufficiently large with respect to $h$.  

    Let $G$ be an $n$-vertex $d$-degenerate graph with disjoint vertex sets $A_0,B$ so that $|A_0|> d^6|B|$ and $e(G[A_0,B])\ge \frac{nd}{10}$. Assume $G$ does not have an induced subgraph $G'$ satisfying any of the following: 
    \begin{itemize}
        \item $G'$ is the 1-subdivision of $K_h$;
        \item $G'$ has $n'\ge d^{1/5}$ vertices with $d(G')\ge (n')^{1-1/(100h)}$;
        \item $G'$ has no $C_4$ and $d(G')\ge d^{1/(50000h)}$.
    \end{itemize}
    
    Then we can find some induced subgraph $G'$ which is the $1$-subdivision of a simple graph $H$ with average degree $d(H)\ge d$. 
\end{proposition}

\begin{remark}
    We note that in the statement of Proposition~\ref{polychotomy Kh} (and Proposition~\ref{polychotomy degree bounded}, stated later), we do not have any assumptions about lacking $K_{s,s}$-subgraphs. These are general results which are true for arbitrary unbalanced graphs $G$. To finish our arguments, we will simply use that any such $G'$ that lacks $K_{s,s}$-subgraphs will have a desirable induced subgraph.
\end{remark}
    
\noindent We delay the proof of Proposition~\ref{polychotomy Kh} to the next section. We are now ready to prove our main theorem.

\begin{proof}[Proof of Theorem~\ref{thm: main1}]
 Fix $h\ge 2$ and consider $s$ sufficiently large.

 Let $G$ have no $K_{s,s}$-subgraphs with $d:=d(G)\ge s^{500h^2}$. By the largeness of $s$, we may assume $d$ is sufficiently large. We want to find an induced proper balanced subdivision of $K_h$. We shall assume that $G$ does not contain an induced copy of the $1$-subdivision of $K_h$, as otherwise we are done.
 
We will argue that we can always pass to an induced subgraph $G'$ satisfying one of the four outcomes:
\begin{itemize}
    \item $G'$ is the 1-subdivision of $K_h$;
    \item $G'$ has $n'\ge d^{1/5}$ vertices with $d(G') \ge (n')^{1-1/(100h)}$;
    \item $G'$ has no $C_4$ and $d(G')\ge d^{1/(62500h)}$;
    \item $G'$ is the $1$-subdivision of a (simple) graph $H$ with $d(H)\ge d$.
\end{itemize}\noindent If the first bullet holds, we obtain a contradiction. On the other hand, if the second bullet holds, we apply Lemma~\ref{one sided erdos hajnal} to $G'$ (using the facts that $G$ has no $K_{s,s}$, $n'\geq d^{1/5}\geq s^{C^*_F}$ and that $\varepsilon_{F}\geq 1/(100h)$, where $F$ is the $1$-subdivision of $K_h$) which brings us back to the first bullet, a contradiction.

Meanwhile, if the third bullet holds, then we find a balanced subdivision of $K_{h'}$ with $h' \ge d^{1/(250000000h)}$ by Proposition~\ref{base case}, which is far greater than $h$ for large $d$. 

Finally, if the fourth bullet holds, we may find an induced balanced subdivision of $K_{h'}\subset G'$ with $h' = \Omega(d^{1/2})$ by Corollary~\ref{subdivision reduction} (taking $s=t=2$ and using the fact $H$ is simple and therefore $G'$ has no $C_4$).

We are going to show now that one of the four bullets must necessarily hold. The result will then follow, as explained above. 
   
We may and will assume as well that $G$ is $d$-degenerate for some $d\geq s^{500h^2}$. First, we apply Lemma~\ref{Lem:dichotomy} (with $L \coloneqq d^6$) to $G$. If the first case in Lemma~\ref{Lem:dichotomy} holds, then we have the exact requirements to apply Proposition~\ref{polychotomy Kh} and therefore find some induced $G'$ satisfying one of the first three bullets with the same $d$.  

Otherwise, the second case happens and we can pass to some induced subgraph $G^*$ which is almost-regular (with $d^*:=d(G^*) \ge d^{1/2}$ and $\Delta(G^*)\le (d^*)^{1+1/(500h^2)}$, as $d^*$ is sufficiently large). Here, we apply Lemma~\ref{Lem: denseorC_4} to find some $G'$ satisfying either the second or third bullet.

 This completes the proof of Theorem~\ref{thm: main1}.
\end{proof}

\subsection{General degree-bounded classes}

We have now reached the final frontier. In the previous subsection, one of our conditions for success was when $G$ contained the $1$-subdivision of $K_h$ as an induced subgraph. Hence we could assume throughout our analysis that $G$ did not contain such an induced subgraph, allowing us to apply Lemma~\ref{one sided erdos hajnal}. Here, we will replace the role of the $1$-subdivision of $K_h$ by another graph $H_k$ which has no $C_4$ and average degree $k$, so that we may assume throughout $G$ is $H_k$-free (or else we will be immediately done). 

\begin{lemma}\label{Lem: Hk}
    For each $k\ge 2$, there exists a bipartite graph $H_k$ on $\le 8k^2$ vertices with $\Delta(H_k)\le 2k$, so that $d(H_k)\ge k$ which contains no $C_4$. Additionally $H_k$ is an induced subgraph of the $1$-subdivision of $K_{8k^2}^{(2k)}$ (the complete $2k$-uniform hypergraph on $8k^2$ vertices).
\end{lemma}
\begin{proof}
    Whenever $k$ is a prime-power, one can construct a bipartite $k$-regular graph $H_k$ on $2k^2$ vertices without a $C_4$ (see e.g. \cite[Theorem~1]{conlon}). The result now follows for general $k$ because we can always find a power of $2$, $q= 2^t$, where $k\le q< 2k$.

    Now for $n,s$, write $\Gamma_{n,s}$ for the 1-subdivision of $K_n^{(s)}$. It remains to show that $H_k$ is an induced subgraph  $\Gamma_{8k^2,2k}$. This follows from two observations, which can be easily verified. 

    \begin{observation}
         $\Gamma_{n,s}$ contains $\Gamma_{n-1,s-1}$ and $\Gamma_{n-1,s}$ as induced subgraphs.
    \end{observation}
    \begin{observation}
        Let $H$ be an $s$-regular bipartite graph with bipartition $A,B$, where $N(y)\neq N(y')$ for distinct $y,y'\in B$. Then $H$ is an induced subgraph of $\Gamma_{|A|,s}$.
    \end{observation}
    Now because $H_k$ is $q$-regular for some $q\in [k,2k)$, it must have distinct neighborhoods (else it would have a $4$-cycle). Also $H_k$ has $2q^2\le 8k^2-(2k-q)$ vertices (since $q<2k$), so we can apply the above observations.
\end{proof}

We now state our modification of Proposition~\ref{polychotomy Kh}. 
\begin{proposition}\label{polychotomy degree bounded}
    Let $D,k \ge 1$, and suppose $d$ is sufficiently large with respect to $D$ and $k$. 

    Let $G$ be an $n$-vertex $d$-degenerate graph with disjoint vertex sets $A_0,B$ such that $|A_0|>d^{D+3}|B|$ and  $e(G[A_0,B]) \ge \frac{nd}{10}$. Moreover, suppose  $G$ does not have an induced subgraph $G'$ satisfying any of the following: 
    \begin{itemize}
        \item $G'$ is the $1$-subdivision of $K_{8k^2}^{(2k)}$;
        \item $G'$ has $n'\ge d^{1/5}$ vertices and $d(G')\ge (n')^{1-1/(200k)}$;
        \item $G'$ has no $C_4$ and $d(G')\ge d^{1/(125000 k)}$.
    \end{itemize}Then, $G$ has an induced subgraph $G'$ with $d(G')\ge D$ that has no $K_{2k,2k}$-subgraphs.
\end{proposition}

This fourth outcome will be combined with the following result the authors together with Du, McCarty and Scott showed in \cite{c4free}.  

\begin{theorem}\label{old degree-bound}
    Let $k,s\geq 2$ and let $G$ be a graph with no $K_{s,s}$. Then, $G$ has an induced subgraph $G'\subset G$ with $d(G')\ge k$, provided $d(G)\ge k^{Cs^3}$, for some absolute constant $C>0$. 
\end{theorem}
\begin{remark}
    It would be enough to have a much weaker quantitative version of the above which is easily deduced from two results proved by McCarty~\cite{McCarty} and Letzter, Kwan, Sudakov, and Tran~\cite{KLST}.
\end{remark} 

\begin{proof}[Proof of Theorem~\ref{thm:main2}]
    Fix $k\ge 2$, and set $D \coloneqq k^{Ck^3}$ as in Theorem~\ref{old degree-bound}. Consider $s$ large. Now let $G$ be a graph with average degree $d:=d(G)\ge s^{5000k^4}$ that contains no $K_{s,s}$-subgraphs. We shall deduce that $G$ has an induced subgraph $G''\subset G$ that has no $4$-cycles with $d(G'')\ge k$. 
    
    We may and will assume that $G$ does not contain the graph $H_k$ from Lemma~\ref{Lem: Hk} (else we are immediately done). 
    
    We now shall show that we can pass to an induced subgraph $G'$ satisfying one of five outcomes:
    \begin{itemize}
        \item $G'$ is $H_k$;
        \item $G'$ is the subdivision of $K_{(8k)}^{2k}$;
        \item $G'$ has $n'\ge d^{1/5}$ vertices with $d(G')\ge (n')^{1-1/(200k)}$;
       \item $G'$ has no $C_4$ and $d(G')\ge d^{1/(125000k)}$;
        \item $G'$ has no $K_{2k,2k}$ (as a subgraph) and $d(G')\ge D$.
    \end{itemize}
    \noindent Clearly, the first outcome would be a contradiction, and the second outcome implies the first (by Lemma~\ref{Lem: Hk}). Moreover, the third outcome also implies the first (via Lemma~\ref{one sided erdos hajnal} with $H:= H_k$, noting that $\varepsilon_{H_k}\geq 1/(200k)$ and $n'\geq d^{1/5}\geq s^{1000k^4}\geq s^{C^*_{H_k}}$). Now, the fourth outcome is exactly what we want with room to spare, since $d^{1/(125000k)}>k$. Finally, if the fifth outcome happens, then by definition of $D$, we can pass to some induced subgraph $G''\subset G'$ that lacks $C_4$'s with $d(G'')\ge k$, as we wanted to show. Therefore, all that is left to do is to show that one of the five bullets must hold.
    
    As before, we may assume $G$ is $d$-degenerate. Next, we apply Lemma~\ref{Lem:dichotomy} to $G$ with $L\coloneqq 4d^{D+3}$. If the first case holds, then we have some heavily unbalanced partition of $G$ satisfying the requirements as in Proposition~\ref{polychotomy degree bounded} and therefore one of the last four bullets must happen. 
    
    Otherwise, the second case holds and we can pass to some induced subgraph $G^*$ which is almost-regular (with $d^*:= d(G^*)=\Omega(d/\log(L)^2) \geq d^{1/2}$ and $\Delta(G^*)\le (d^*)^{1+1/(1000k)}$, since $d$ is large enough). We may then finish off by invoking Lemma~\ref{Lem: denseorC_4} (taking $\eps:= 1/(1000k)$) to find some $G'$ satisfying either the third or fourth bullet. This completes the proof.
\end{proof}

We mention that essentially the same argument gives the following result, which will be more convenient in future applications.

\begin{proposition}
    Fix a bipartite graph $H$, along with some $\eps>0$ and $t\ge 1$.

    Now consider an $H$-free graph with average degree $d$, where $d$ is sufficiently large. Then there exists an induced subgraph $G'\subset G$ satisfying either:
    \begin{itemize}
        \item $G'$ has $n'\ge d^{1/10}$ vertices and $d(G')\ge (n')^{1-\eps}$;
        \item $G'$ lacks $4$-cycles and $d(G')\ge t$.
    \end{itemize}
\end{proposition}
\noindent We quickly sketch a proof.
\begin{proof}[Proof (sketch)]
    Since $H$ is bipartite, one can pick $h$ so that $H$ is an induced subgraph of the $1$-subdivision of $K_{h}^{(h/2)}$. So we may assume $G$ is $\Gamma$-free, where $\Gamma$ is the $1$-subdivision of $K_{h}^{(h/2)}$. We also assume that $G$ contains no induced subgraph $G'\subset G$ on $n'\ge d(G)^{1/10}$ vertices with $d(G')\ge (n')^{1-\eps}$.
    
    We lastly set $D =  t^{Ch^3}$, as in Theorem~\ref{old degree-bound}.

    Now, suppose $d:=d(G)$ is sufficiently large (and that $G$ is $d$-degenerate). 
    We apply Lemma~\ref{Lem:dichotomy} with $L := 4d^{D+3}$. 
    
    If the first case holds, then instead of invoking Proposition~\ref{polychotomy degree bounded}, we now use the more general Propositions~\ref{unbalanced 2 clean} and \ref{unbalanced 2 messy} with parameters $h:= h,k:=h/2, D:= D, \eps':=\eps/5$. Supposing neither of our assumptions are contradicted, we can find some induced $G'\subset G$ with no $4$-cycles and $d(G')\ge t$, either by applying Theorem~\ref{old degree-bound} or by the fact that $d^{\eps'/125}\ge t$ (since $d$ is large). So we are done in this case.

    Otherwise, the second case of Lemma~\ref{Lem:dichotomy} holds and we can pass to some almost-regular subgraph $G^*\subset G$. We are then done by an application of Lemma~\ref{deletion}.
\end{proof}

\section{Proof of the polychotomy propositions}

\subsection{Big picture}

The purpose of this section is to establish Propositions~\ref{polychotomy Kh} and \ref{polychotomy degree bounded}. They will both follow from a single general argument.

We first need a ``model result'', in the same vein as to Lemma~\ref{Lem:unbalanced}. This is how we intend to proceed in both of the aforementioned results when none of the ``forbidden bullets'' occur.

\begin{lemma}\label{unbalanced 2 clean}
    Let $d,D,k\ge 2$ and suppose $d$ is sufficiently large with respect to $D$ and $k$.  
    
    Let $G$ be a graph which is $d$-degenerate with a bi-partition into disjoint vertex sets $A,B$ so that the following holds:
    \begin{itemize}
        \item $|A|>d^{D+1}|B|$;
        \item $G[A] = \emptyset$;
        \item $d(x)\le 100 d$ for each $x\in A$.
    \end{itemize}
    Furthermore, suppose that for every $x\in A$, there is an independent set of size $D$, $I_x:=\{y_1,\dots,y_D\}\in \binom{N(x)}{D}$, so that $|I_{x}\cap N(x')|<k$ for all $x'\in A\setminus \{x\}$.

    Then, $G$ contains an induced subgraph $G'$ which is the $1$-subdivision $H$ of some $D$-uniform multihypergraph $\HH$ with $d(\HH)\ge d$ and with the property that $H$ has no $K_{k,k}$.
\end{lemma}
\begin{remark}
    The first three bullets are essentially what we get after applying Claim~\ref{make regular} (in the proof of Lemma~\ref{Lem:unbalanced}). The crucial condition is the existence of the independent sets $I_x$, which guarantee us that $G'$ does not have any copy of $K_{k,k}$.
\end{remark}
\begin{proof}
    Let $I_x$ be as in the statement above.
    Since $G$ is $d$-degenerate, we can direct the edges of $G[B]$ to obtain a digraph $Q$ with maximum out-degree at most $ d$.

    Now, let $B'$ be a random subset of $B$ where each vertex is included in $B'$ with probability $p:= 1/(1000Dd)$ (independently). Let $A'$ be the set of $x\in A$ for which $N(x)\cap B' = I_x$ and $B''$ be the set of $y\in B'$ with $|N_Q^+(y)\cap B'| = 0$. Finally, let take $A'':= \{x\in A': N(x)\cap B'' = N(x)\cap B'\}$.

    It is clear that $\E[|B''|]\le \E[|B'|] = |B|p$. Meanwhile, for a fixed $a\in A$, we have that $\P(a\in A') = p^D(1-p)^{d(x)-D}\ge p^D(1-p(100d))\ge p^D/2$ and furthermore
    \[\P(x\in A''|x\in A') \ge 1-\sum_{y\in I_x}\P(y\not\in B''|x\in A') \ge 1-Ddp\ge 1/2.\] 
    Using that $G[I_x]=\emptyset$. Therefore, we have $\E[|A''|] \ge |A|p^D/4\geq |B|$ (using that $d$ is sufficiently large). 
    
    Whence $\E[|A''|-d|B''|]\ge 0$ and so there is some outcome with $|A''|\geq d|B''|$. 
    Taking $G' := G[A''\cup B'']$ gives the desired graph. Indeed, clearly this is the 1-subdivision of some $s$-uniform hypergraph $\HH$ with $d(\HH)\geq d$. Observe that by assumption we have $|N_{G'}(x)\cap N_{G'}(x')|\le |I_x\cap N(x')|<s$ for distinct $x,x'\in A''$ which will imply $G'$ cannot have a $K_{s,s}$, as we wanted to show.  
\end{proof}

\noindent To finish, we show that we can always reduce to this case, unless one of the three forbidden bullets happens.

\begin{proposition}\label{unbalanced 2 messy}
    Let $h,k,D \ge 2$ along with an $\eps\in (0,1/10)$, and suppose $d$ is sufficiently large with respect to $h,k$ and $D$.

    Let $G$ be an $n$-vertex $d$-degenerate graph with disjoint vertex sets $A_0,B$ so that $|A_0|> 100d^{D+3}|B|$ and $e(G[A_0,B])\ge \frac{nd}{10}$. Suppose $G$ does not have an induced subgraph $G'$ satisfying any of the following: 
    \begin{itemize}
        \item $G'$ is the 1-subdivision of $K_h^{(k)}$;
         \item $G'$ has $n'\ge d^{1/5}$ vertices with $d(G')\ge (n')^{1-5\eps}$;
        \item $G'$ has no $C_4$ and $d(G')\ge d^{\eps/125}$.
    \end{itemize}
    
    \noindent Then, we can find $A\subset A_0$ so that $A,B$ satisfies all the conditions of Lemma~\ref{unbalanced 2 clean} (with the same choice of $D,k$).
\end{proposition}
\noindent We delay the details to the upcoming subsection (Subsection~\ref{Grandfinale}). We now quickly discuss how to deduce our polychotomies propositions. 
\begin{proof}[Proof of Proposition~\ref{polychotomy Kh}]Apply Propositions~\ref{unbalanced 2 clean} and \ref{unbalanced 2 messy} with $h:=h, k:=2,D:=2,\eps:= \frac{1}{500h}$. If none of the three forbidden bullets hold then we may pass to an induced subgraph $G'\subset G$, which lacks $4$-cycles and is the $1$-subdivision of a multihypergraph $H$ with $d(H)\ge d$. 

Since $G'$ has no $C_4$, we have that $H$ must be a simple graph (since if we had multiple edges in $H$, then this would create a $C_4$ inside $G'$). Thus, we may now apply Corollary~\ref{subdivision reduction} to find a balanced subdivision of $K_{h'}$ inside $G'$ where $h'\ge \Omega(\sqrt{d})$. Assuming $d$ is sufficiently large, we get $h'\ge h$, giving the desired result.
\end{proof}

\begin{proof}[Proof of Proposition~\ref{polychotomy degree bounded}]
    Exactly as above, apply Propositions~\ref{unbalanced 2 clean} and \ref{unbalanced 2 messy}  to $G$ with $h:=8k^2, k:=2k,D:=D,\eps := \frac{1}{1000k}$. If none of the three bullets hold, then we may as before pass to an induced subgraph $G'\subset G$ which has no $K_{k,k}$ and which is the $1$-subdivision of a $D$-uniform hypergraph $\HH$ with $d(\HH)\ge d$.

    Since $d(\HH)$ is larger than $s$, it is not hard to see that $d(G')\ge D$ ($\HH$ will have more edges than vertices, so $G'$ is a bipartite graph where all vertices in the larger part have degree $D$).
\end{proof}

\subsection{Establishing the last proposition}\label{Grandfinale}

This section is devoted to proving Proposition~\ref{unbalanced 2 messy}. After appealing to results like Claim~\ref{make regular} and Lemma~\ref{Lem: denseorC_4}, we will morally reduce to the following problem. Consider $x\in A_0$ with $N(x)\subset B$ being an independent set of size $d$, and with $|N(x)\cap N(x')|$ not being ``too large'' for any other $x'\in A_0\setminus \{x\}$. Here, we would like to find some $I_x\in \binom{N(x)}{D}$ where $|I_x\cap N(x')|<k$ for each $x'\in A_0\setminus \{x'\}$, or otherwise find an induced subgraph $G'$ isomorphic to the $1$-subdivision of $K_h^{(s)}$. 

We handle the task of finding $I_x$ or a copy of $K_h^{(s)}$ via a ``shattering lemma'' (Lemma~\ref{shattering lemma}), which is inspired by the notion of VC-dimension. Although a more quantitative shattering result follows by applying the well-known Sauer-Shelah lemma, we choose to give a simple self-contained proof in the spirit of Lemma~\ref{one sided erdos hajnal}.

Given a family $\mathcal{F}\subset \mathcal{P}([n])$, we say that $\mathcal{F}$ $k$-shatters a set $R\subset [n]$ if for every $S \in \binom{R}{k}$, there is $F\in \mathcal{F}$ such that $R\cap F=S$. 

We start with an easy preliminary observation.
\begin{lemma}\label{pre shatter}
    Let $r\ge k\ge 2$ and $n \ge 10 r^2$, and suppose $\FF\subset \PP([n])$ has the property that for every $S\in \binom{[n]}{k}$, there exists some $F\in \FF$ with $F\supset S$. Furthermore, suppose that $|F|\le n/r^{k+1}$ for every $F\in \FF$. Then, there is some $r$-subset $R\subset [n]$ which is $k$-shattered by $F$. 
\end{lemma}
\begin{proof}
    For each subset $S\in \binom{[n]}{k}$, we assign some $F_S \in \FF$ so that $S\subset F_S$ (if there are multiple choices, pick one arbitrarily). 

    Now, let $x_1,\dots,x_r\in [n]$ be chosen uniformly at random (with repetitions allowed). We have that $\P(x_1,\dots,x_r\text{ are all distinct})\ge 1- \binom{r}{2}\frac{1}{n}\ge 9/10$. 

    Next, for $e\in \binom{[r]}{k}$, we bound the probability of the event, $\EE_{e}$, that $(x_i)_{i\in e}$ are all distinct and $x_{\ell}\in F_{\{x_i:i\in e\}}$ for some other $\ell\in [r]\setminus e$. Since $|F_{\{x_i:i\in e\}}|\le n/r^{k+1}$, and each $x_i$ is chosen uniformly at random, a union bound gives $\P(\EE_{e}) \le (r-k)/r^{k+1}\le 1/r^k$. Summing over the $e$, we have that $\P(\EE_{e} \text{ holds for some }e)\le \binom{r}{k}\frac{1}{r^k}\le 1/2$. 

    Thus, with positive probability $x_1,\dots,x_r$ are all distinct, and none of the events $\EE_e$ happen. Which means that $R:=\{x_1,\dots,x_r\}$ is a set of size $r$ which is $s$-shattered by $\FF$.
\end{proof}

We now simply apply the hypergraph Ramsey theorem (which says that for every $k,h,D$, there exists some finite $N$ so that every $2$-coloring of $\binom{[N]}{k}$ either has a $h$-set $R$ where $\binom{R}{k}$ is monochromatic in the first color, or a $D$-set $B$ where $\binom{B}{k}$ is monochromatic in the second color) to finish.

\begin{lemma}\label{shattering lemma}
    Let $D,h,k\ge 2$. Then, there exists $\delta:= \delta(D,h,s)>0$ so that the following holds for all sufficiently large $n$.

    Let $\FF\subset \PP([n])$ where $F<\delta n$ for each $F\in \FF$, and which does not $k$-shatter a set of size $h$. Then there exists a set $I\subset [n]$ of size $D$ where $|I \cap F|<k$ for all $F\in \FF$. 
\end{lemma}
\begin{proof}
    Let $N$ be the 2-color Ramsey number of $K_h^{(k)}$ and $K_D^{(k)}$ and let $\delta := 1/N^{k+1}$. 

    Now consider a family $\FF\subset \PP([n])$ where $|F|<\delta n$. Define $\GG\subset \binom{[n]}{k}$ to be the set of all $S\in \binom{[n]}{k}$ such that $S\not \subset F$ for all $F\in \FF$. As $n$ is large, we have that $\FF\cup \GG$ satisfies the assumptions of Lemma~\ref{pre shatter} (with $k:=k,r:=N$), thus we can find some set $X$ of size $N$ which is $k$-shattered by $\FF \cup \GG$. 
    
    For each $k$-subset $S\in \binom{X}{k}$, we color it red if $S = F\cap X$ for some $F\in \FF$, and otherwise we color it blue. Note that if $S$ is colored blue, we must have that $S\in \GG$. By our choice of $N$, we can either find a set $R\in \binom{X}{h}$ where all $k$-subsets are red (which implies that $\FF$ $s$-shatters a set of size $h$), or a set $B\in \binom{X}{D}$ where all $k$-subsets are blue. By assumption $R$ does not exist, so we are done taking $I:= B$ (we cannot have $|I\cap F|\ge k$ for any $F\in \FF$, as this would mean $B$ contains a non-blue $s$-subset).
\end{proof}

Finally we deduce Proposition~\ref{unbalanced 2 messy}.

\begin{proof}[Proof of Proposition~\ref{unbalanced 2 messy}]
    We apply first Claim~\ref{make regular} so we may find some $A'\subset A_0$ with $|A'|>3 d^{D+2}|B|$, so that $G[A'] = \emptyset$ and $d_B(x) \in [d/20,10d]$ for each $x\in A'$. Thus, the requirements in the bullet points in order to apply Lemma~\ref{unbalanced 2 clean} are satisfied. All that is left is to find the appropriate choices for $I_x$. Actually, we will have to pass to a subset $A\subset A'$ with $|A|\ge |A'|/(3d)$ and find the $I_x$, for all $x\in A$. 
    
    To find the $I_x$, we shall use the following result.
    \begin{claim}\label{find unique pair}
        Let $x\in A'$ be any vertex. Suppose we cannot find an induced subgraph $G'$ satisfying any of the three bullets (in Proposition~\ref{unbalanced 2 messy}), then there must exist some independent set $I_x = \{y_1,\dots,y_D\}\in \binom{N_B(x)}{D}$ where $\#\{x'\in A' \setminus \{x\}: |N(x')\cap I_x|\ge s\} <d$.
    \end{claim}
    \noindent We shall prove the above claim in a moment. But first, we quickly show how to deduce our result using this claim.

    For each $x\in A'$, we can fix some $I_x$ given by Claim~\ref{find unique pair}. Now create an auxiliary graph $H$ on $A'$, where $x\sim x'$ if $|I_x\cap N(x')|\ge s$ and/or $|I_{x'}\cap N(x)|\ge s$. We have that $\Delta(H)<2d$, so we can greedily find an independent set $A\subset V(H) = A'$ with $|A|\ge |A'|/(2d+1)\ge |A'|/(3d)$. 
    
    This choice of $A$ has all the desired properties. Indeed, all the bulleted assumptions from Proposition~\ref{unbalanced 2 clean} still hold. Meanwhile for $x\in A$ with $I_x = \{y_1,\dots,y_D\}$, we have $\{x'\in A\setminus \{x\}: |N(x')\cap I_x|\ge s\} \subset N_H(x)\cap A =\emptyset$ since $H[A] = \emptyset$ (so the required conditions about the $I_x$ are satisfied).

    It remains to prove Claim~\ref{find unique pair}.
    \begin{proof}[Proof of Claim~\ref{find unique pair}]
        Fix $x\in A'$, and assume none of the aforementioned $G'$ exist. By construction of $A'$, we have that $d_B(x)\ge d/20$.
    
        First, we claim there must be an independent set $J=J_x\subset N(x)$ with $|J|=\sqrt{d/20}$. Assuming otherwise, then applying Lemma~\ref{independent set or dense} to $G[N(x)]$ with $\eps:=\eps$, we get some induced subgraph $G'\subset G[N(x)]$ satisfying either the second or third forbidden bullet (since $d_B(x)\ge d/20> d^{4/5}$ assuming $d$ is large).
        
        We now seek to find $I_x$ within $J$. Let $S=S_x$ be the set of vertices $x'\in A_0\setminus \{x\}$ with $d_J(x')\ge |J|^{1-\eps}$, whose neighborhood correlates heavily with $J$. We must have $|S|<|J|$, otherwise we can pass to an induced subgraph $G'\subset G[J\cup S]$ on $2|J|\ge d^{1/5}$ vertices with average degree $\ge |J|^{1-\eps/2}$, which satisfies the second forbidden bullet (assuming $2\le n^{\eps/10}$, which holds when $d$ is large).

        We shall now find some $I_x\in \binom{J}{D}\subset \binom{N(x)}{D}$ where $\{x'\in A': |N(x')\cap I_x|\ge s\} \subset \{x\}\cup S$. Since $|S|<|J|<d$ (and recalling $J$ is an independent set), this will complete the proof.

        Write $\FF:= \{N(x')\cap N(x): x'\in A'\setminus (\{x\}\cup S)\}$. Since $d$ is sufficiently large, we have that $|F|<\delta |J|$ for every $F\in \FF$ (where $\delta = \delta(h,D,s)>0$ is the constant from Lemma~\ref{shattering lemma}). Now assuming (for the sake of contradiction) that no such $I_x$ exists, then by Lemma~\ref{shattering lemma}, $\FF$ will $s$-shatter a set of size $h$ (assuming $d$ is sufficiently large so that $|J|$ is sufficiently large). 
        
        This implies that there are distinct vertices $y_1,\dots,y_h\in J$ so that for each $e\in \binom{[h]}{s}$, there exists some $x_e \in A'$ where $N(x_e) \cap \{y_1,\dots,y_h\} = \{y_i:i\in e\}$. Taking $G' := G[\{y_1,\dots,y_h\}\cup \{x_e:e\in \binom{[h]}{s}\}]$, we have that $G'$ is the $1$-subdivision of $K_h^{(s)}$ (since $J,A'$ are both independent sets within $G$). So this $s$-shattered set would mean we have a $G'$ satisfying the first forbidden bullet, which can't happen. So the choice of $I_x$ must exist.
    \end{proof}
\end{proof}

\section{Concluding Remarks}
We have shown that every degree-bounded class of graphs is polynomially degree-bounded. 
It could be interesting to classify all degree-bounded classes with a degree bounding function which grows linearly. 

Now, regarding $\chi$-bounded classes of graphs our knowledge is much more limited. Recall, we do not even know whether for every $T$, the class of $T$-free graphs is $\chi$-bounded. This has only been verified for very specific trees (see e.g. \cite{survey}) which includes all paths. 
We conjecture that actually for every path $P$, the class of $P$-free graphs is \textit{polynomially} $\chi$-bounded. 
\begin{conjecture}\label{path conjecture}
    For every $k$, there is $C_k>0$ such that every graph $G$ with $\chi(G)\geq \omega(G)^{C_k}$ contains an induced path of length $k$. 
\end{conjecture}
Scott, Seymour and Spirkl~\cite{SSSp5} showed recently that the class $\CC$ of $P_5$-free graphs is $\chi$-bounded where $f(x)=x^{\log_2(x)}$ is a $\chi$-binding function for $\CC$. 

Recall Theorem~\ref{thm: main} states that if the average degree of a graph $G$ is sufficiently large compared with its clique number then $G$ must either contain an induced complete bipartite graph or an induced subdivision of every small graph but we cannot ensure either of the structures by itself. 

But what happens if we replace average degree with chromatic number? We observe again that neither of these structures could individually be forced, as there are graphs of arbitrarily high chromatic number and girth at least $5$ and triangle-free graphs with arbitrarily high chromatic number and no induced subdivision of the $1$-subdivision of $K_5$ \cite{counterex}. 
Scott and Seymour~\cite{bananas} showed the family of graphs without an induced subdivision of a $(\ell,t)$-theta graph i.e. a graph consisting of two vertices joined by $t$-vertex disjoint paths of length at least $\ell$ is $\chi$-bounded. We conjecture that actually avoiding the family of long balanced subdivisions of $K_{2,t}$ is enough to guarantee $\chi$-boundedness. 
\begin{conjecture}
    For every $\ell$ and $t$, there is $f_{\ell,t}: \mathbb{N} \rightarrow \mathbb{N} $ such that every graph $G$ without an induced copy of a $t$-theta graph where all paths are of the same length at least $\ell$ has $\chi(G)\leq f_{\ell,t}(\omega(G))$. 
\end{conjecture}

Maybe $f_{\ell,t}$ could even be taken to be polynomial, but this is not known even for cycles (the case of $t=2$), since a positive answer would resolve Conjecture~\ref{path conjecture}.

\subsection{Acknowledgements}
The authors would like to thank Alex Scott for helpful remarks. 

\hide{\subsection{Further questions}

\begin{conjecture}
    Let $H$ be a bipartite graph. Does there exist some $C_H$ where the following holds for $s\ge 1$. 

    For every $H$-free graph $G$ with $d(G)\ge s^{C_H}$, we either have:
    \begin{itemize}
        \item $G$ contains a $K_{s,s}$;
        \item or there is an induced subgraph $G'\subset G$ with $d(G')\ge s$ that contains no $4$-cycles.
    \end{itemize}
\end{conjecture}
Write $\tau(G)$ to denote the maximal $s$ so that $G$ contains a $K_{s,s}$, and $\kappa(G)$ to denote the maximal $k$ so that there is an induced subgraph $G'\subset G$ with $d(G')\ge k$ where $G'$ has no $4$-cycles. The above says that if $G$ is $H$-free with $d(G)\ge d$, then $\max\{\tau(G),\kappa(G)\} \ge d^{\Omega_H(1)}$. 

With some minor adjustments, the methods from our paper easily confirm this conjecture when $H= K_{s,t}$ or $H$ is the 1-subdivision of $K_h$. 

By more aggressively applying our dichotomy result (Lemma~\ref{Lem:dichotomy}),

Our methods can prove that we either have $\tau(G)\ge d^{\Omega_H(1)}$ or $\kappa(G)\ge \omega_{d\to \infty}(1)$ (second condition boring).

The above claims that we for any $H$-free family of graphs, $\FF$, we have $\max\{\tau(G)$}

\bibliographystyle{abbrv}
\bibliography{bibliography}

\end{document}